\documentclass[12pt,a4paper]{amsart}
\calclayout
\usepackage[utf8]{inputenc}
\usepackage{hyperref}
\usepackage{mathrsfs}
\usepackage{mathtools}
\mathtoolsset{showonlyrefs}
\usepackage{stackrel}
\usepackage{comment}
\newtheorem{thm}{Theorem}[section]
\newtheorem{lem}[thm]{Lemma}
\newtheorem{prop}[thm]{Proposition}
\newtheorem{cro}[thm]{Corollary}

\theoremstyle{definition}

\newtheorem{exa}[thm]{Example}
\newtheorem{asp}{Assumption}
\numberwithin{equation}{section}
\allowdisplaybreaks
\begin{document}
\title
[Superprocesses with stable branching]
{\large Limit theorems for a class of critical superprocesses with stable branching}
\author[Y.-X. Ren, R. Song and Z. Sun]{Yan-Xia Ren, Renming Song and Zhenyao Sun}
\address
{Yan-Xia Ren\\
	School of Mathematical Sciences\\
	Peking University\\
	Beijing, P. R. China, 100871}
\email{yxren@math.pku.edu.cn}
\thanks{The research of Yan-Xia Ren is supported in part by NSFC (Grant Nos. 11671017  and 11731009), and LMEQF.
}
\address
{Renming Song\\
	Dept of Mathematics\\
	University of Illinois at Urbana-Champaign\\
	Urbana, IL, USA, 61801}
\email{rsong@illinois.edu}
\thanks{The research of Renming Song is supported in part by the Simons Foundation (\#429343, Renming Song).}
\address
{Zhenyao Sun\\
School of Mathematics and Statistics\\
Wuhan University\\
Wuhan, Hubei, P. R. China, 430072}
\email{zhenyao.sun@pku.edu.cn}
\thanks{Zhenyao Sun is the corresponding author.}
\begin{abstract}
  We consider a critical superprocess $\{X;\mathbf P_\mu\}$ with general spatial motion and spatially dependent stable branching mechanism with lowest stable index $\gamma_0 > 1$. We first show that, under some conditions, $\mathbf P_{\mu}(\|X_t\|\neq 0)$ converges to $0$ as $t\to \infty$ and is regularly varying with index $(\gamma_0-1)^{-1}$. Then we show that, for a large class of non-negative testing functions $f$, the distribution of $\{X_t(f);\mathbf P_\mu(\cdot|\|X_t\|\neq 0)\}$, after appropriate rescaling, converges weakly to a positive random variable $\mathbf z^{(\gamma_0-1)}$ with Laplace transform $E[e^{-u\mathbf z^{(\gamma_0-1)}}]=1-(1+u^{-(\gamma_0-1)})^{-1/(\gamma_0-1)}.$
\end{abstract}
\maketitle
\section{Introduction}
\subsection{Background}
The study of the asymptotic behaviors of critical branching particle systems has a long history.
It is well known that for a critical Galton-Watson process $\{(Z_n)_{n\geq 0}; P\}$, we have
\begin{align}
  \label{eq: Kolmogorov's result with finite variance}
	n P(Z_n > 0)
	\xrightarrow[n\to \infty]{} \frac{2}{\sigma^2}
\end{align}
and
\begin{align}
\label{eq: Yaglom's result with finite variance}
	\Big\{ \frac{Z_n}{n}; P(\cdot| Z_n > 0) \Big\}
	\xrightarrow[n \to \infty]{\operatorname{law}} \frac{\sigma^2}{2} \mathbf e,
\end{align}
where $\sigma^2$ is the variance of the offspring distribution and $\mathbf e$ is an exponential random variable with mean $1$.
The result \eqref{eq: Kolmogorov's result with finite variance} is due to Kolmogorov \cite{Kolmogorov1938Zur-losung}, and the result \eqref{eq: Yaglom's result with finite variance} is due to Yaglom \cite{Yaglom1947Certain}.
For further references to these results, see \cite{Harris2002The-theory, KestenNeySpitzer1966The-Galton-Watson}.
Since then, lots of analogous results have been obtained for more general critical branching processes with finite 2nd moment, see \cite{AsmussenHering1983Branching, AthreyaNey1974Functionals, AthreyaNey1972Branching, JoffeSpitzer1967On-multitype} for example.

Notice that \eqref{eq: Kolmogorov's result with finite variance} and \eqref{eq: Yaglom's result with finite variance} are still valid when  $\sigma^2 = \infty$,
see \cite{KestenNeySpitzer1966The-Galton-Watson} for example.
In this case, the limits in \eqref{eq: Kolmogorov's result with finite variance} and \eqref{eq: Yaglom's result with finite variance} are degenerate, and thus more appropriate scalings are needed.
Research in this direction was first conducted by Zolotarev \cite{Zolotarev1957More} in a simplified continuous time set-up, which is then extended by Slack \cite{Slack1968A-branching} to discrete time critical Galton-Watson processes allowing infinite variance.
The main result of  \cite{Slack1968A-branching} can be stated as follows. 
Consider a critical Galton-Watson process $\{(Z_n)_{n\geq 0}; P\}$.
Assume that the  generating function $f(s)$ of the offspring distribution is of the form
\begin{align}
  \label{eq: offspring generating function with alpha moment}
	f(s)
	= s + (1-s)^{1+ \alpha} l(1-s),
	\quad s\geq 0,
\end{align}
where $\alpha\in (0, 1]$ and $l$ is a function slowly varying at $0$.
Then
\begin{align} 
  \label{eq: extinction probability of critical GW process without 2rd moment}
	P(Z_n > 0)
	= n^{-1/\alpha} L(n),
\end{align}
where $L$ is a function slowly varying at $\infty$, and
\begin{align} 
  \label{eq: conditional distribution of critical GW process without 2rd moment}
	\big\{ P(Z_n > 0) Z_n; P(\cdot | Z_n > 0)\big\}
	\xrightarrow[n\to \infty]{\operatorname{law}} \mathbf z^{(\alpha)},
\end{align}
where $\mathbf z^{(\alpha)}$ is a positive random variable with Laplace transform
\begin{align}
  \label{e:zloltarev}
	E[e^{- u \mathbf z^{(\alpha)}}]
	= 1 - (1+ u^{-\alpha})^{-1/\alpha},
	\quad u \geq 0.
\end{align}
In \cite{Slack1972Further}, Slack also considered the converse of this problem:
In order for $\big\{ P(Z_n > 0) Z_n; P(\cdot | Z_n > 0)\big\}$ to have a non-degenerate weak limit, the generating function of the offspring distribution must be of the form of \eqref{eq: offspring generating function with alpha moment} for some $0 < \alpha \leq 1$.
For shorter and more unified approaches to these results, we refer our readers to \cite{Borovkov1989Method, Pakes2010Critical}.

Goldstein and Hoppe \cite{GoldsteinHoppe1978Critical} considered the asymptotic behavior of multitype critical Galton-Watson processes without the 2nd moment condition. 
Their main result can be stated as follows.
Let $\mathbf Z_n=(Z_n^{(1)}, \dots, Z_n^{(d)})$ be a $d$-type, nonsingular Galton-Watson process with its mean matrix $M:= (E[Z_1^{(j)}| Z_0^{(i)} = 1, Z_0^{(k)} = 0, \forall k \neq i])_{1\leq i,j\leq d}$ being positive regular, that is, all entries of $M$ are finite and there exists a number $n \geq 1$ such that all entries of $M^n$ are positive.
Denote by $\mathbf F(\mathbf s) = (\mathbf F_1(\mathbf s), \dots, \mathbf F_d(\mathbf s))$ the generating function of the offspring distribution, and by $\mathbf F^{(n)}(\mathbf s), ~ n>1,$ its $n$th iterates.
Assume that the process is critical in the sense that the maximal eigenvalue of $M$ is $1$.
Let $\mathbf v$ and $\mathbf u$ be the left and right  eigenvectors of $M$, respectively, corresponding to this maximal eigenvalue $1$, and normalized so that $\mathbf v \cdot \mathbf u = 1$ and $\mathbf 1 \cdot \mathbf u = 1$, where $\mathbf 1$ is the vector $(1,\dots, 1)$.
Suppose that
\begin{align}
  \label{eq: regularly varying condition for multitype branching process}
	\mathbf v G(\mathbf 1-x\mathbf u) \mathbf u
	= x^\alpha l(x),
	\quad x > 0,
\end{align}
where $0 < \alpha \leq 1$;
$l$ is slowly varying at $0$;
and the matrix $G(\mathbf s)$ is defined by
\begin{align}
	\mathbf 1 - \mathbf F(\mathbf s)
	= (M - G(\mathbf s))(\mathbf 1 - \mathbf s),
	\quad \mathbf s \in \mathbb R_+^d.
\end{align}
Let $a_n := \mathbf v \cdot (\mathbf 1 - \mathbf F^{(n)}(\mathbf 0))$, with $\mathbf 0 \in \mathbb R_+^d$ being the vector $(0,\dots, 0)$.
It was shown in \cite{GoldsteinHoppe1978Critical} that, for each $\mathbf i \in \mathbb N_0^d \setminus \{\mathbf 0\}$,
\begin{align}
  \label{eq: limit behavior of the exitinction probability without finite variance of multitype branching processes}
  n l(a_n)
	\operatorname{P}(\mathbf Z_n \neq \mathbf 0| Z_0 = \mathbf i)^\alpha
	\xrightarrow[n\to \infty]{}
	\frac{(\mathbf i \cdot \mathbf u)^\alpha}\alpha,
\end{align}
and for each $\mathbf j \in \mathbb N_0^d$,
\begin{align}
  \label{eq: conditioned normalized multitype branching process}
	\{ a_n \mathbf Z_n \cdot \mathbf j ; P(\cdot | \mathbf Z_n \neq \mathbf 0, \mathbf Z_0 = \mathbf i)\}
	\xrightarrow[n\to \infty]{\operatorname{law}} (\mathbf v\cdot \mathbf j) \mathbf z^{(\alpha)},
\end{align}
where $\mathbf z^{(\alpha)}$ is a random variable with Laplace transform given by \eqref{e:zloltarev}.
For the converse of this problem, Vatutin \cite{Vatutin1977Limit} showed that in order for the left side of \eqref{eq: conditioned normalized multitype branching process} to have a non-degenerate weak limit, one must have \eqref{eq: regularly varying condition for multitype branching process} for some $0 < \alpha \leq 1$.
Vatutin \cite{Vatutin1977Limit} also considered analogous results for continuous time multitype critical Galton-Watson processes.
	
Asmussen and Hering \cite[Sections~6.3~and~6.4]{AsmussenHering1983Branching} discussed similar questions for critical branching Markov processes $(Y_t)$ in a general space $E$ under some ergodicity condition (the so-called condition (M), see \cite[p.~156]{AsmussenHering1983Branching}) on the mean semigroup of $(Y_t)$.
When the second moment is infinite, under a condition  parallel  to \eqref{eq: regularly varying condition for multitype branching process} (the so-called condition (S) \cite[p.~207]{AsmussenHering1983Branching}), results parallel to \eqref{eq: limit behavior of the exitinction probability without finite variance of multitype branching processes} and \eqref{eq: conditioned normalized multitype branching process} were proved in \cite[Theorem~6.4.2]{AsmussenHering1983Branching} for critical branching Markov processes.

In this paper, we are interested in a class of measure-valued branching Markov processes known as $(\xi, \psi)$-superprocesses:
$\xi$, the spatial motion of the superprocess, is a Hunt process on a locally compact separable metric space $E$;
$\psi$, the branching mechanism of the superprocess, is a function on $E \times [0,\infty)$ of the form
\begin{align} 
  \label{eq: branching mechanism}
	\psi(x,z):=
	- \beta(x) z + \sigma (x)^2 z^2 + \int_{(0,\infty)} (e^{-zy} - 1 + zy) \pi(x,dy),
	\quad x\in E, z\geq 0,
\end{align}
where $\beta, \sigma \in \mathscr B_b(E)$ and $\pi(x,dy)$ is a kernel from $E$ to $(0,\infty)$ such that $\sup_{x\in E} \int_{(0,\infty)} (y\wedge y^2) \pi(x,dy) < \infty$.
For the precise definition and properties of superprocesses, see \cite{Li2011Measure-valued}.

Results parallel to \eqref{eq: Kolmogorov's result with finite variance} and \eqref{eq: Yaglom's result with finite variance} have been obtained for some critical superprocesses by Evans and Perkins \cite{EvansPerkins1990Measure-valued} and Ren, Song and Zhang \cite{RenSongZhang2015Limit}.
Evans and Perkins \cite{EvansPerkins1990Measure-valued} considered critical superprocesses with branching mechanism of the form $(x,z)\mapsto z^2$ and with the spatial motion satisfying some ergodicity conditions.
Ren, Song and Zhang \cite{RenSongZhang2015Limit} extended the results of \cite{EvansPerkins1990Measure-valued} to a class of critical superprocesses with general branching mechanism and general spatial motions.
The main results of \cite{RenSongZhang2015Limit} are as follows.
Let $\{(X_t)_{t\geq 0}; \mathbf P_\mu \}$ be a critical superprocess starting from a finite measure $\mu$ on $E$.
Suppose the spatial motion $\xi$ is intrinsically ultracontractive with respect to some reference measure $m$, and the branching mechanism $\psi$ satisfies the following second moment condition
\begin{align}
  \label{eq: second moment condition}
	\sup_{x\in E} \int_{(0,\infty)} y^2 \pi(x,dy)
	< \infty.
\end{align}
For any finite measure $\mu$ on $E$ and any measurable function $f$ on $E$, we use $\langle f,\mu\rangle$ to denote the integral of $f$ with respect to $\mu$.
Put $\|\mu\|=\langle 1, \mu\rangle$.
Under some other mild assumptions, it was proved in  \cite{RenSongZhang2015Limit} that
\begin{align}
  \label{eq: Kolmogorov type result with 2rd moment}
	t \mathbf P_\mu(\|X_t\| \neq 0)
	\xrightarrow[t\to \infty]{} c^{-1}
  \langle \phi, \mu \rangle,
\end{align}
and for a large class of testing functions $f$ on $E$,
\begin{align}\label{eq: Ygalom type result with 2rd moment}
	\{ t^{-1}X_t(f); \mathbf P_\mu (\cdot |\|X_t\| \neq 0)\}
	\xrightarrow[t\to \infty]{\operatorname{law}} c \langle \phi^*, f\rangle_m \mathbf e.
\end{align}
Here, the constant $c > 0$ is independent of the choice of $\mu$ and $f$;
$\langle\cdot, \cdot \rangle_m$ denotes the inner product in $L^2(E, m)$;
$\mathbf e$ is an exponential random variable with mean $1$;
and $\phi$ (respectively, $\phi^*$) is the principal eigenfunction of (respectively, the dual of) the generator of the mean semigroup of $X$.
In \cite{RenSongSun2017Spine}, we provided an alternative probabilistic approach to \eqref{eq: Kolmogorov type result with 2rd moment} and \eqref{eq: Ygalom type result with 2rd moment}.

It is natural to ask whether results parallel to \eqref{eq: extinction probability of critical GW process without 2rd moment} and \eqref{eq: conditional distribution of critical GW process without 2rd moment} are still valid for some critical superprocesses without the second moment condition \eqref{eq: second moment condition}.
A simpler version of this question has already been answered in the context of continuous-state branching processes	(CSBPs) which can be viewed as superprocesses without spatial movements.
Kyprianou and Pardo \cite{Kyprianou2008Continuous} considered CSBPs $\{(Y_t)_{t\geq 0}; P\}$ with stable branching mechanism $\psi(z) =c z^\gamma$, where $c > 0$ and $\gamma \in (1,2]$.
They showed that for all $x\geq 0$, with $c_t := (c(\gamma - 1)t)^{1/(\gamma - 1)}$,
\begin{align} 
  \label{eq: conditional limit of CSBP with stable branching}
	\{c_t^{-1}Y_t; P( \cdot |Y_t > 0,Y_0 = x)\}
	\xrightarrow[t\to \infty]{\operatorname{law}} \mathbf z^{(\gamma - 1)},
\end{align}
where $\mathbf z^{(\gamma - 1)}$ is a random variable with Laplace transform given by \eqref{e:zloltarev} (with $\alpha=\gamma-1$) and is independent of the initial position $x$.
Recently, Ren, Yang and Zhao \cite{RenYangZhao2014Conditional} studied CSBPs $\{(Y_t)_{t\ge 0}; P\}$ with branching mechanism
\begin{align}
  \label{eq: regular varing of branching mechanism of a CSBP}
	\psi(z)
	= c z^\gamma l(z),
	\quad z\geq 0,
\end{align}
where $c > 0$, $\gamma \in (1,2]$ and $l$ is a function slowly varying  at $0$.
It was proved in  \cite{RenYangZhao2014Conditional} that for all $x\geq 0$, with $\lambda_t: = P_1(Y_t > 0)$,
\begin{align}
  \label{eq: conditional limit of CSBP}
	\{ \lambda_t Y_t ; P(\cdot | Y_t > 0, Y_0 = x)\}
  \xrightarrow[t\to \infty]{\operatorname{law}} \mathbf z^{(\gamma - 1)}
\end{align}
where the distribution of the random variable $\mathbf z^{(\gamma - 1)}$ is given by \eqref{e:zloltarev} (with $\alpha=\gamma-1$) and is independent of the initial position $x$.

Later, Iyer, Leger and Pego \cite{IyerLegerPego2015Limit} considered the converse problem:
Suppose  $\{(Y_t)_{t\geq 0}; P\}$ is a CSBP with critical branching mechanism $\psi$ satisfying Grey's condition.
In order for the left side of \eqref{eq: conditional limit of CSBP} to have a non-trivial weak limit for some positive constants $(\lambda_t)_{t\geq 0}$, one must have \eqref{eq: regular varing of branching mechanism of a CSBP} for some $1< \gamma \leq 2$.

In this paper, we will establish a result parallel to \eqref{eq: conditional limit of CSBP with stable branching} for some critical $(\xi,\psi)$-superprocesses $\{X; \mathbf P\}$ with spatially dependent stable branching mechanism.
In particular, we assume that the spatial motion $\xi$ is intrinsically ultracontractive with respect to some reference measure $m$, and the branching mechanism takes the form
\begin{align}
	\psi(x,z)
  &= - \beta(x) z + \kappa(x) \int_0^\infty (e^{-z y} - 1+ z y) \frac{dy}{\Gamma(- \gamma(x)) y^{1+ \gamma(x)}}
  \\&=  -\beta (x) z + \kappa(x) z^{\gamma(x)},
	\quad x\in E, z \geq 0,
\end{align}
where $\beta \in \mathscr B_b(E)$, $\gamma \in \mathscr B^+_b(E)$, $\kappa \in \mathscr B^+_b(E)$ with $1< \gamma(\cdot )<2$, $\gamma_0 := \operatorname{ess\,inf}_{m(dx)} \gamma(x)> 1$ and $\operatorname{ess\,inf}_{m(dx)}\kappa(x) > 0$.
Let $\mu$ be an arbitrary finite initial measure on $E$.
We will show that $\mathbf P_{\mu}( \| X_t\| \neq 0)$ converges to $0$ as $t\to \infty$ and is regularly varying at infinity with index $\frac{1}{\gamma_0 - 1}$.
Furthermore, if $m(x: \gamma(x) = \gamma_0)>0$, we will show that
\begin{align}
	\lim_{t\to \infty}\eta^{-1}_t \mathbf P_{\mu}( \|X_t\| \neq0)
	= \mu(\phi),
\end{align}
and for a large class of non-negative testing functions $f$,
\begin{align}
  \label{eq: result2}
	\{ \eta_t X_t(f) ; \mathbf P_{\mu}(\cdot | \|X_t\|\neq 0) \}
	\xrightarrow[t\to \infty]{\operatorname{law}}
	\langle f, \phi^*\rangle_m \mathbf z^{(\gamma_0 - 1)},
\end{align}
where $\eta_t := \big( C_X(\gamma_0 - 1) t \big)^{- \frac {1} {\gamma_0 - 1} },$ $C_X := \langle \mathbf 1_{\gamma(\cdot) = \gamma_0} \kappa\cdot \phi^{\gamma_0}, \phi^* \rangle_m$ and $\mathbf z^{(\gamma_0 - 1)}$ is a random variable with Laplace transform given by \eqref{e:zloltarev} (with $\alpha=\gamma_0-1$).
Notice that the distribution of the weak limit $\langle f, \phi^*\rangle_m \mathbf z^{(\gamma_0 - 1)}$ does not depend on $\mu$.
Precise statements of the assumptions and the results are presented in the next subsection.
It is interesting to mention here that, even though the stable index $\gamma(x)$ is spatially dependent, the limiting behavior of the critical superprocess $\{X; \mathbf P\}$ depends primarily on the lowest index $\gamma_0$.

\subsection{Model and results}

We first fix our notation.
Unless stated explicitly otherwise, $E$ is assumed to be a locally compact separable metric space. 
We use $\mathscr B(E)$ to denote the collection of all  Borel subsets of $E$ and also the collection of all  Borel functions on $E$. Define $\mathscr B_b(E) :=\{f \in \mathscr B(E): \sup_{x\in E}|f(x)|<\infty \}$, $\mathscr B^+(E) :=\{f\in \mathscr B(E): \forall x\in E,~f(x)\geq 0\}$ and $\mathscr B^{++}(E) :=\{f\in \mathscr B(E): \forall x\in E,~f(x)> 0\}$.
Define $\mathscr B^+_b(E) := \mathscr B_b(E) \cap \mathscr B^+(E)$ and $\mathscr B^{++}_b(E):= \mathscr B_b(E) \cap \mathscr B^{++}(E)$.
Denote by $\mathcal M_E$ the collection of all Borel measures on $E$.
Denote by $\mathcal M^\sigma_E$ the collection of all  $\sigma$-finite Borel measures on $E$.
For simplicity, we write $\mu(f)$ and sometimes $\langle \mu, f\rangle$ for the integration of a function $f$ with respect to a measure $\mu$.
For any $f \in \mathscr B^+(E)$, define $\mathcal M^f_E:= \{\mu \in \mathcal M_E: \mu(f) < \infty\}$. 
In particular, $\mathcal M^1_E$ is the collection of all  finite Borel measures on $E$.

We now give the definition of a $(\xi, \psi)$-superprocess:
Let the spatial motion $\xi=\{(\xi_t)_{t\geq 0};(\Pi_x)_{x\in E}\}$ be an $E$-valued Hunt process with its lifetime denoted by $\zeta$, and the branching mechanism $\psi$ be a function on $E\times[0,\infty)$ given by \eqref{eq: branching mechanism}.
We say an $\mathcal M^1_E$-valued Hunt process $X=\{(X_t)_{t\geq 0}; (\mathbf P_\mu)_{\mu \in \mathcal M^1_E}\}$ is a \emph{$(\xi,\psi)$-superprocess} if for each $t\geq 0, \mu \in \mathcal M_E^1$ and  $f\in \mathscr B^+_b(E)$, we have
\begin{align}
	\mathbf P_\mu [e^{-X_t(f)}] = e^{-\mu(V_tf)},
\end{align}
where the function $(t,x) \mapsto V_tf(x)$ on $[0,\infty) \times E$ is the unique locally bounded positive solution to the equation
\begin{align}
  \label{eq:FKPP_in_definition}
	V_t f(x) + \Pi_x \Big[  \int_0^{t\wedge \zeta} \psi (\xi_s,V_{t-s} f) ds \Big]
	= \Pi_x [ f(\xi_t)\mathbf 1_{t<\zeta} ],
	\quad t \geq 0, x \in E.
\end{align}
(In this paper, for any real-valued function $F$ on $E\times [0,\infty)$ and real-valued function $f$ on $E$, we write $F(x,f):= F(x,f(x))$ for simplicity.)

Recall that the branching mechanism $\psi$ is given by \eqref{eq: branching mechanism} and its linear coefficient $\beta$ is a bounded Borel function on $E$. Define the \emph{Feynman-Kac semigroup}
\begin{align}
  P^\beta_tf(x)
  := \Pi_x \big[e^{\int_0^{t} \beta(\xi_r)dr} f(\xi_t)\mathbf 1_{t<\zeta}\big],
  \quad t\geq 0, x\in E, f\in \mathscr B_b(E).
\end{align}
It is known, see \cite[Proposition 2.27]{Li2011Measure-valued} for example, $(P^\beta_t)$ is the \emph{mean semigroup} of the superprocess $\{X; \mathbf P\}$ in the sense that
\begin{align}
  \label{eq:mean_formula}
	\mathbf P_\mu [X_t(f)]
	= \mu(P^\beta_t f),
	\quad \mu \in \mathcal M^1_E,
	t \geq 0,f \in \mathscr B_b(E).
\end{align}
The mean semigroup plays a central role in the study of the asymptotic behavior of superprocesses.
As discussed in \cite{EvansPerkins1990Measure-valued}, in order to have a result like \eqref{eq: Ygalom type result with 2rd moment} or \eqref{eq: result2}, we have to establish the asymptotic behavior of the mean semigroup first.
This can be done under the following assumptions on the spatial motion $\xi$:
\begin{asp} 
  \label{asp: 1}
  There exist an $m \in \mathcal M_E^\sigma$ with full support on the state space $E$ and a family of strictly positive,
	bounded continuous functions $\{ p_t(\cdot,\cdot): t > 0 \}$ on $E \times E$ such that
  \begin{align}
    \Pi_x[ f(\xi_t)\mathbf 1_{t < \zeta} ]
    = \int_E p_t(x,y) f(y) m(dy),
    &\quad t>0, x \in E,f \in \mathscr B_b(E);
    \\\int_E p_t(y,x)m(dy)
    \leq 1,	
    &\quad t>0,x\in E;
    \\\int_E \int_E p_t(x,y)^2 m(dx) m(dy)
    <\infty,
    &\quad t> 0;
  \end{align}
	and the functions $x \mapsto \int_E p_t(x,y)^2 m(dy)$ and $x \mapsto \int_E p_t(y,x)^2 m(dy)$ are both continuous.
\end{asp}

We will write $\langle f, g\rangle_m$ for $\int_E fg dm$ to emphasize that it is the inner product in the Hilbert space $L^2(E, m)$.
Let $(P_t^*)_{t\geq 0}$ be the dual of the transition semigroup $(P_t)_{t\geq 0}$ in $L^2(E, m)$, i.e.,
\begin{align}
  P_0^* = I;
  \quad P_t^* f(x) := \int_E p_t(y,x) f(y) m(dy),
  \quad t > 0, x\in E, f\in \mathscr B_b(E).
\end{align}
Under Assumption \ref{asp: 1}, it is proved in \cite{RenSongZhang2015Limit} and \cite{RenSongZhang2017Central} that $(P_t)_{t \geq 0}$ and $(P^*_t)_{t \geq 0}$ are both strongly continuous semigroups of compact operators on $L^2(E,m)$.
Let $\widetilde L$ and $\widetilde L^*$ be the generators of the semigroups $(P_t)_{t \geq 0}$ and $(P^*_t)_{t \geq 0}$, respectively.
Denote by $\sigma(\widetilde L)$ and $\sigma(\widetilde L^*)$ the spectra of $\widetilde L$ and $\widetilde L^*$, respectively.
According to \cite[Theorem V.6.6]{Schaefer1974Banach}, $\widetilde \lambda := \sup \text{Re}(\sigma(\widetilde L)) = \sup \text{Re}(\sigma(\widetilde L^*))$ is a common eigenvalue of multiplicity $1$ for both $\widetilde L$ and $\widetilde L^*$.
By the argument in \cite{RenSongZhang2015Limit}, the eigenfunctions $\widetilde \phi$ of $\widetilde L$ and $\widetilde \phi^*$ of $\widetilde L^*$ associated with the eigenvalue $\widetilde \lambda$ can be chosen to be strictly positive and continuous everywhere on $E$.
We further normalize $\widetilde \phi$ and $\widetilde \phi^*$ by $\big\langle \widetilde \phi, \widetilde \phi \big\rangle_m = \big\langle \widetilde \phi, \widetilde \phi^*\big\rangle_m = 1$ so that they are unique.

It is also proved in \cite{RenSongZhang2015Limit, RenSongZhang2017Central} that there exists a function $p^\beta_t(x,y)$ on $(0,\infty) \times E \times E$ which is continuous in $(x,y)$ for each $t>0$ such that
\begin{align}
	e^{-\|\beta\|_\infty t} p_t(x,y)
	\leq p^{\beta}_t(x,y)
	\leq e^{\|\beta\|_\infty t} p_t(x,y),
	\quad t>0, x, y\in E,
\end{align}
and that for any $t>0, x\in E$ and $f \in \mathscr B_b(E)$,
\begin{align}
	P^\beta_t f(x)
	= \int_E p_t^\beta (x,y) f(y) m(dy).
\end{align}
$(p^\beta_t)_{t\geq 0}$ is called the \emph{density} of the semigroup $(P^\beta_t)_{t\geq 0}$.
Define the dual semigroup $(P^{\beta *}_t)_{t \geq 0}$ by
\begin{align}
	P^{\beta *}_0 = I;
	\quad P^{\beta *}_t f(x)
	:= \int_E p^\beta_t (y,x) f(y) m(dy),
	\quad t>0, x\in E, f\in \mathscr B_b(E).
\end{align}
It is proved in \cite{RenSongZhang2015Limit, RenSongZhang2017Central} that $(P^\beta_t)_{t \geq 0}$ and $(P^{\beta *}_t)_{t \geq 0}$ are both strongly continuous semigroups of compact operators on $L^2(E,m)$.
Let $L$ and $L^*$ be the generators of the semigroups $(P^\beta_t)_{t \geq 0}$ and $(P^{\beta *}_t)_{t \geq 0}$, respectively.
Denote by $\sigma(L)$ and $\sigma(L^*)$ the spectra of $L$ and $L^*$, respectively.
According to \cite[Theorem V.6.6]{Schaefer1974Banach}, $\lambda := \sup \text{Re}(\sigma(L)) = \sup \text{Re}(\sigma(L^*))$ is a common eigenvalue of multiplicity $1$ for both $L$ and $L^*$.
By the argument in \cite{RenSongZhang2015Limit}, the eigenfunctions $\phi$ of $L$ and $\phi^*$ of $L^*$ associated with the eigenvalue $\lambda$ can be chosen to be strictly positive and continuous everywhere on $E$.
We further normalize $\phi$ and $\phi^*$ by $\langle\phi, \phi\rangle_m = \langle\phi,\phi^*\rangle_m = 1$ so that they are unique.
Moreover, for each $t\geq 0$ and $x\in E$, we have $P^\beta_t \phi(x) = e^{\lambda t} \phi(x)$ and $P^{\beta *}_t \phi^*(x) = e^{\lambda t} \phi^*(x)$.
We refer to $\phi$ (resp. $\phi^*$) and $\lambda$ as the \emph{principal eigenfunction} and the \emph{principal eigenvalue} of $L$ (resp. $L^*$).

Now, from
\begin{align}
	\mathbf P_\mu[X_t(\phi)]
	= e^{\lambda t} \mu(\phi),
	\quad t \geq 0,
\end{align}
we see that, if $\lambda > 0$, the mean of $X_t(\phi)$ will increase exponentially; if $\lambda < 0$, the mean of $X_t(\phi)$ will decrease exponentially; and if $\lambda = 0$, the mean of $X_t(\phi)$ will be a constant.
Therefore, we say $X$ is \emph{supercritical, critical or subcritical}, according to $\lambda > 0$, $\lambda = 0$ or $\lambda < 0$, respectively.
Since we are only interested in the critical case, we assume the following:
\begin{asp} 
  \label{asp: 2}
  The superprocess $X$ is critical, i.e., $\lambda = 0$.
\end{asp}

Our next assumption is on the spatial motion $\xi$:

\begin{asp} 
  \label{asp: 3}
	$\widetilde \phi$ is bounded, and $(P_t)_{t\geq 0}$ is \emph{intrinsically ultracontractive}, that is, for each $t>0$, there is a constant $c_t >0$ such that for each $x,y\in E$, $p_t(x,y) \leq c_t \widetilde \phi(x) \widetilde \phi^*(y)$.
\end{asp}
	
Under Assumption \ref{asp: 3}, it is proved in \cite{RenSongZhang2015Limit, RenSongZhang2017Central} that the principal eigenfunction $\phi$ of the Feynman-Kac semigroup $(P^\beta_t)_{t\geq 0}$ is also bounded.
Moreover, $(P^\beta_t)_{t\geq 0}$ is also \emph{intrinsically ultracontractive}, in the sense that for each $t>0$, there is a constant $c_t >0$ such that for each $x,y\in E$, $p^\beta_t(x,y) \leq c_t \phi(x) \phi^*(y)$.
In fact, it is proved in \cite{KimSong2008Intrinsic} that for each $t>0$, $(p^\beta_t(x,y))_{x,y\in E}$ is comparable to $(\phi(x)\phi^*(y))_{x,y\in E}$ in the sense that there is a constant $c_t > 1$ such that
\begin{align}
  \label{eq: p-t-beta is comparable to phi phi-star}
	c_t^{-1}
	\leq \frac {p^\beta_t(x,y)} {\phi(x)\phi^*(y)}
	\leq c_t,
	\quad x,y \in E.
\end{align}
It is also shown in \cite{KimSong2008Intrinsic} that there are constants $c_0, c_1 > 0$ such that
\begin{align}
  \label{eq:q(t,x,y)}
	\sup_{x,y\in E} \Big|\frac{p^\beta_t(x,y)}{\phi(x)\phi^*(y)} - 1 \Big|
	\leq c_0 e^{-c_1 t},
	\quad t > 1.
\end{align}
Assumption \ref{asp: 3} is a pretty strong assumption on the semigroup $\{P_t : t \ge 0\}$.
For example, it rules out the semigroup of Brownian motion on $\mathbb R^d$ and the semgroup of Ornstein-Uhlenbeck process on $\mathbb R^d$.
However, this assumption is satisfied in a lot of cases.  
In \cite{RenSongZhang2015Limit},  a list of examples of processes satisfying 
Assumptions \ref{asp: 1} and \ref{asp: 3} were given. 
For the convenience of our readers, we will briefly recall some of these examples in Subsection \ref{sec:examples}.

Recall that the branching mechanism is given by \eqref{eq: branching mechanism}. 
We assume the following:
\begin{asp} 
  \label{asp: 4}
	The branching mechanism $\psi$ is of the form:
  \begin{align}
    \psi(x,z)
    &= - \beta(x) z + \kappa(x) \int_0^\infty (e^{-z y} - 1+ z y) \frac{dy}{\Gamma(- \gamma(x)) y^{1+ \gamma(x)}}
    \\&= -\beta (x) z + \kappa(x) z^{\gamma(x)},
    \quad x\in E, z \geq 0,
  \end{align}
  where $\gamma \in \mathscr B^+_b(E)$, $\kappa \in \mathscr B^{++}_b(E)$ with $1< \gamma(\cdot )<2$, $\gamma_0 := \operatorname{ess\,inf}_{m(dx)} \gamma(x)> 1$ and $\kappa_0:=\operatorname{ess\,inf}_{m(dx)}\kappa(x) > 0$.
\end{asp}
Here we used the definition of the Gamma function on the negative half line:
\begin{align}\label{eq: definition of Gamma function}
  \Gamma(x)
  := \int_0^\infty t^{x-1} \Big(e^{-t} - \sum_{k=0}^{n-1} \frac{(-t)^k}{k!}\Big) dt,
  \quad -n< x< -n+1, n\in \mathbb N.
\end{align}
	
We now present the main result of this paper:
\begin{thm}
  \label{thm: main theorem}
  Suppose that $\{(X_t)_{t\geq 0}; (\mathbf P_\mu)_{\mu \in \mathcal M_E^1}\}$ is a $(\xi, \psi)$-superprocess satisfying Assumptions \ref{asp: 1}--\ref{asp: 4}.
  Then,
  \begin{itemize}
  \item[(1)]
    $\{X; \mathbf P\}$ is non-persistent, that is, for each $t > 0$ and $x\in E$,
    $\mathbf P_{\delta_x}( \| X_t\| = 0) > 0$.
  \item[(2)]
    For each $\mu \in \mathcal M^1_E$, $\mathbf P_{\mu}(\|X_t\| \neq 0)$ converges to $0$ as $t \to \infty$ and is regularly varying at infinity with index $-(\gamma_0-1)^{-1}$.
    Furthermore, if $m(x: \gamma (x)= \gamma_0)>0$, then
    \begin{align}
      \lim_{t\to\infty} \eta_t^{-1}\mathbf P_{\mu}(\|X_t\| \neq 0)
      =\mu(\phi).
    \end{align}
  \item[(3)]
    Suppose $m( x:\gamma(x)=\gamma_0 )>0$.
    Let $f \in \mathscr B^+(E)$ be such that $\langle f, \phi^* \rangle_m > 0$ and $\| \phi^{-1}f \|_\infty < \infty$. Then for each $\mu \in \mathcal M_E^1$,
    \begin{align}
      \{\eta_t X_t(f) ; \mathbf P_{\mu}(\cdot |\|X_t\| \neq 0) \}
      \xrightarrow[t\to \infty]{\operatorname{law}} \langle f, \phi^*\rangle_m \mathbf z^{(\gamma_0 - 1)}.
    \end{align}
  \end{itemize}
  Here, $\eta_t := \big( C_X(\gamma_0 - 1) t \big)^{- \frac {1} {\gamma_0 - 1} }$, $C_X := \langle \mathbf 1_{\gamma(\cdot) = \gamma_0} \kappa\cdot \phi^{\gamma_0}, \phi^* \rangle_m$ and $\mathbf z^{(\gamma_0 - 1)}$ is a random variable with Laplace transform given by \eqref{e:zloltarev} (with $\alpha=\gamma_0 -1$).
\end{thm}

\subsection{Methods and overview}
To establish Theorem \ref{thm: main theorem}(2) and Theorem \ref{thm: main theorem}(3), we use a spine decomposition theorem for 	$X$.
Roughly speaking, the spine is the trajectory of an immortal moving particle and the spine decomposition theorem says that, after a martingale change of measure, the transformed superprocess can be decomposed in law as the sum of a copy of the original superprocess and a measure-valued immigration process along this spine, see \cite{EckhoffKyprianouWinkel2015Spines, EnglanderKyprianou2004Local, LiuRenSong2009Llog}.
The martingale used for the change of measure is $(e^{-\lambda t} X_t(\phi))_{t\geq 0}$.
Under Assumptions \ref{asp: 1} and \ref{asp: 3}, the spine process $\{\xi; \Pi^{(\phi)}\}$ is an ergodic process.
We take advantage of this ergodicity to study the asymptotic behavior of the superprocess.

Similar idea has already been used by Powell \cite{Powell2015An-invariance} to establish results parallel to \eqref{eq: Kolmogorov type result with 2rd moment} and \eqref{eq: Ygalom type result with 2rd moment} for a class of critical branching diffusion processes.
Let $\{(Y_t)_{t\geq 0}; P\}$ be a branching diffusion process in a bounded domain with finite second moment.
As have been discussed in \cite{Powell2015An-invariance}, a direct study of the partial differential equation satisfied by the survival probability $(t,x) \mapsto P_{\delta_x}(\|Y_t\| \neq 0)$ is tricky.
Instead, by using a spine decomposition approach, Powell \cite{Powell2015An-invariance} showed that the survival probability decays like $a(t)\phi(x)$, where $\phi(x)$ is the principal eigenfunction of the mean semigroup of $(Y_t)$ and $a(t)$ is a function capturing the uniform speed.
Then the problem is reduced to the study of a single ordinary differential equation satisfied by $a(t)$.
Later, inspired by \cite{Powell2015An-invariance}, we gave in \cite{RenSongSun2017Spine} a similar proof of \eqref{eq: Kolmogorov type result with 2rd moment} for a class of general critical superprocesses with finite second moment.
In this paper, we will  generalize these arguments to a class of general critical superprocesses without finite second moment and establish Theorem \ref{thm: main theorem}(2).
For the conditional weak convergence result, i.e., Theorem \ref{thm: main theorem}(3), we use a fact that the Laplace transform given in \eqref{e:zloltarev} can be characterized by a non-linear delay equation (see Lemma \ref{lem: characterize the general Mittag-Leffler distribution}).
Using the spine method, we show that the Laplace transform of the one-dimensional distributions of the superprocess, after a proper rescaling, can be 	characterized by a similar equation (see \eqref{eq: equation for normalized V_T}).
Then, the desired convergence of the distributions can be established by a comparison between the equations.	
Again, the ergodicity of the spine process plays a central role in the comparison.

A similar idea for establishing weak convergence through a comparison of the equations satisfied by the distributions has already been used by us in \cite{RenSongSun2017A-2-spine, RenSongSun2017Spine}.
We characterized the exponential distribution using its double size-biased transform; and to help us make the comparison, we investigated the double size-biased transform of the corresponding processes.
However, the double-size-biased transform of a random variable requires its second moment being finite.
Since we do not assume the second moment condition in this paper, we can not use the method of double size-biased transform.

In \cite{Powell2015An-invariance} (for critical branching diffusions in a bounded domain with finite variance) and in \cite{RenSongSun2017Spine, RenSongZhang2015Limit} (for general critical superprocesses with finite variance), the conditional weak convergence was proved in two steps.
First, a convergence result was established for $\phi$, the principal eigenfunction of the mean semigroup of the corresponding process, and then the second moment condition was used to extend the result to more general testing functions.
However, in the present case, since we are not assuming the second moment condition, this type of argument does not work.
Instead, we use a generalized spine decomposition theorem, which is developed in \cite{RenSongSun2017Spine}, to establish Theorem \ref{thm: main theorem}(3) for a large class of general testing functions in one stroke.

The rest of this paper is organized as follows:
In Subsections \ref{sec: Asymptotic equivalence}, \ref{sec: Regularly variation} and \ref{sec: Superprocesses}, we give some preliminary results about the asymptotic equivalence, regularly variation and superprocesses, respectively.
In Subsection \ref{sec: Spine decompositions}, we present the generalized spine decomposition theorem.
In Subsection \ref{sec: Ergodicity}, we discuss the ergodicity of the spine process.
In Subsections \ref{sec: proof of result 1} and \ref{sec: proof of result 2} we give the poofs of Theorem \ref{thm: main theorem}(1) and \ref{thm: main theorem}(2), respectively.
In Subsection \ref{sec: conditional distribution}, we give the equation that characterize the one-dimensional distributions.
In Subsection \ref{sec: proof of result 3}, we give the proof of Theorem \ref{thm: main theorem}(3).
In Appendix \ref{sec: Characterizing the Zolotarev's distribution using an non-linear delay equation}, we give the equation that characterizes the distribution with Laplace transform \eqref{e:zloltarev}, which is used in the proof of Theorem \ref{thm: main theorem}(3).

\section{Preliminaries}
\label{sec: Preliminaries}
\subsection{Asymptotic equivalence} 
\label{sec: Asymptotic equivalence}
In this subsection, we give a lemma on asymptotic equivalence.	
Let $t_0 \in [-\infty,\infty]$.
In this subsection, $(E, \mathscr E)$ is assumed to be a  measurable space.
For any $f_0, f_1\in \mathscr B^{++}({\mathbb R})$, we say $f_0$ and $f_1$ are \emph{asymptotically equivalent at $t_0$}, if $\big|\frac{f_0(t)}{f_1(t)} - 1\big| \xrightarrow[t\to t_0]{} 0$; 
and in this case, we write $f_0(t) \stackrel[t\to t_0]{}{\sim} f_1(t)$.
For any strictly positive measurable functions $g_0, g_1$ on $\mathbb R\times E$, we say $g_0$ and $g_1$ are \emph{uniformly asymptotically equivalent at $t_0$}, if $\sup_{x\in E}\big|\frac{g_0(t,x)}{g_1(t,x)} - 1\big| \xrightarrow[t\to t_0]{} 0$; and in this case, we write $g_0(t,x)\stackrel[t\to t_0]{x\in E}{\sim}g_1(t,x)$.

\begin{lem} 
  \label{lem: asymptotic equivalent of integration}
	Suppose that $f_0,f_1$ are bounded strictly positive measurable functions on $\mathbb R \times E$ and $f_0(t,x)\stackrel[t\to t_0]{x\in E}{\sim}f_1(t,x)$.
  If $\rho$ is a finite non-degenerate measure on $(E, \mathscr E)$, then
  \begin{align}
    \int_E f_0(t,x)\rho(dx)
    \stackrel[t\to t_0]{}{\sim}
    \int_E f_1(t,x)\rho(dx).
  \end{align}
\end{lem}
\begin{proof}
	Since
  \begin{align}
    &\Big| \frac{	\int_E f_0(t,x) \rho(dx) }{ 	\int_E f_1(t,x) \rho(dx)  } - 1 \Big|
      = \Big| \int_E \frac{f_0(t,x)}{f_1(t,x)} \frac{f_1(t,x) \rho(dx)}{	\int_E f_1(t,y) \rho(dy)  } - 1\Big|
    \\&\quad \leq \int_E \Big|  \frac{f_0(t,x)}{f_1(t,x)} - 1 \Big| \frac{f_1(t,x) \rho(dx)}{	\int_E f_1(t,y) \rho(dy)  }
    \leq \sup_{x\in E} \Big|  \frac {f_0(t,x)} {f_1(t,x)} - 1 \Big|
    \xrightarrow[t\to t_0]{} 0,
  \end{align}
  the assertion is valid.
\end{proof}

\subsection{Regular variation}
\label{sec: Regularly variation}
In this subsection, we give some preliminary results on regular variation.
We refer the reader to \cite{BinghamGoldieTeugels1989Regular} for more results on  regular variation.
For $f\in \mathscr B^{++}((0,\infty))$ and $\alpha \in (- \infty, \infty)$, we say $f$ is regularly varying at $\infty$ (resp. at $0$) with index $\alpha$ if for any $u \in (0,\infty)$,
\begin{align}
	\lim_{t\to\infty}\frac{f(ut)}{f(t)}
  = u^\alpha
	\quad \Big(\text{resp. } \lim_{t\to 0}\frac{f(u t)}{f(t)}
	= u^\alpha\Big).
\end{align}
In this case we write  $f\in \mathcal R^\infty_\alpha$ (resp. $f\in \mathcal R^0_\alpha$).
Further, if $\alpha = 0$, then we say $f$ is slowly varying.
According to \cite[Theorem 1.3.1]{BinghamGoldieTeugels1989Regular}, if $L$ is a function slowly varying at $\infty$, then it can be written in the form
\begin{align}
	L(t)
	= c(t) \exp\Big\{\int_{t_0}^t \epsilon(u) \frac{du}{u}\Big\},\quad t\geq t_0,
\end{align}
for some $t_0>0$, where $(c(t))_{t\geq t_0}$ and $(\epsilon(t))_{t\geq t_0}$ are measurable functions with $c(t) \xrightarrow[t\to \infty]{} c \in (0,\infty)$ and $\epsilon(t) \xrightarrow[t\to \infty]{} 0$.
In particular, we know that, there is $t_0 > 0$ large enough such that $L$ is locally bounded on $[t_0,\infty)$.

\begin{lem}[{\cite[Propositions 1.5.8 and 1.5.10]{BinghamGoldieTeugels1989Regular}}]
  \label{lem: exchange slowly varying function and integration}
  Suppose that $L\in \mathcal R^\infty_0$.
  \begin{itemize}
  \item
    Let $t_0\in (0,\infty)$ be large enough so that $L$ is locally bounded on $[t_0,\infty)$. If $\alpha>0 $, then
    \begin{align}
      \int_{t_0}^t L(u)du^\alpha
      \stackrel[t\to \infty]{}{\sim} t^\alpha L(t).
    \end{align}
  \item
    If $\alpha< 0$ then $\int_t^\infty L(u) du^\alpha < \infty$ for $t$ large enough, and
    \begin{align}
      -\int_t^\infty L(u)du^\alpha
      \stackrel[t\to \infty]{}{\sim} t^\alpha L(t).
    \end{align}
  \end{itemize}
\end{lem}

\begin{cro}
  \label{cro: power law and ingetration}
  Suppose that $l\in \mathcal R^0_0$.
  \begin{itemize}
  \item
    Let $s_0\in (0,\infty)$ be small enough so that $l$ is locally bounded on $(0,s_0]$.
    If $\alpha < 0$, then
    \begin{align}
      -\int_s^{s_0} l(u)du^\alpha
      \stackrel[s\to 0]{}{\sim} s^{\alpha} l(s).
    \end{align}
  \item
    If $\alpha > 0$, then $\int_0^s l(u)du^\alpha<\infty$ for $s$ small enough, and
    \begin{align}
      \int_0^s l(u)du^\alpha
      \stackrel[s\to 0]{}{\sim} s^{\alpha} l(s).
    \end{align}
  \end{itemize}
\end{cro}	

\begin{proof}
	Since $l \in \mathcal R^0_0$, we know that, if one defines $L(t):=l(t^{-1})$ for each $t\in (0,\infty)$, then $ L \in \mathcal R^\infty_0$.
	Therefore, there exists $t_0\in (0,\infty)$ such that $L$ is locally bounded on $[t_0,\infty)$.
	Taking $s_0:= t_0^{-1}$, we then immediately get that $l$ is locally bounded on $(0,s_0]$.
	If $\alpha<0 $, then according to Lemma \ref{lem: exchange slowly varying function and integration}, we have
  \begin{align}
    \int_{t_0}^t L(u)du^{-\alpha}
    \stackrel[t\to \infty]{}{\sim} t^{-\alpha}  L(t).
  \end{align}
	Replacing $t$ with $s^{-1}$, we have
  \begin{align}
    -\int_{s}^{s_0} l(u)du^{\alpha}
    =\int_{s_0^{-1}}^{s^{-1}} L(u)du^{-\alpha}
    \stackrel[s\to 0]{}{\sim}  (s^{-1})^{-\alpha}L(s^{-1})
    =s^\alpha l(s),
  \end{align}
	as desired.
	The second assertion can be proved similarly.
\end{proof}

\begin{lem}[{\cite[Theorem 1.5.12]{BinghamGoldieTeugels1989Regular}}] 
  \label{lem: regularly variation and inverse}
	If $f \in \mathcal R_\alpha^\infty$ with $\alpha > 0$, there exists $g \in \mathcal R^\infty_{1/\alpha}$ with
  \begin{align}
    g(f(t))
    \stackrel[t\to \infty]{}{\sim}
    f(g(t))
    \stackrel[t\to \infty]{}{\sim}
    t.
  \end{align}
	Here $g$ is determined uniquely up to asymptotic equivalence as $t\to \infty$.
\end{lem}
\begin{cro} 
  \label{cro: regularly varing and inverse with alpha < 0}
	If $f \in \mathcal R_\alpha^0$ with $\alpha < 0$, there exists $g \in \mathcal R^\infty_{1/\alpha}$ with
  \begin{align} 
    \label{eq: asymptotic inverse with alpha < 0}
    g(f(t))
    \stackrel[t\to 0]{}{\sim}
    t;
    \quad
    f(g(t))
    \stackrel[t\to \infty]{}{\sim}
    t.
  \end{align}
	Here $g$ is determined uniquely up to asymptotic equivalence as $t\to \infty$.
\end{cro}
\begin{proof}
	Since $f \in \mathcal R_\alpha^0$, we know that $\widetilde f \in \mathcal R_{-\alpha}^\infty$ with $\widetilde f(t):= f(t^{-1})$.
	Noticing that $-\alpha > 0$, according to Lemma \ref{lem: regularly variation and inverse}, there exists $h \in \mathcal R_{-1/\alpha}^{\infty}$ such that
  \begin{align} 
    \label{eq: inverse between h and f}
    h(\widetilde f(t))
    \stackrel[t\to \infty]{}{\sim}
    t;
    \quad
    \widetilde f(h(t))
    \stackrel[t\to \infty]{}{\sim}
    t.
  \end{align}
	Denoting by $g := h^{-1} \in \mathcal R_{1/\alpha}^\infty$, the above translates to \eqref{eq: asymptotic inverse with alpha < 0}.

	Now, suppose that there is another $g_0 \in \mathcal R_{1/\alpha}^\infty$ satisfies \eqref{eq: asymptotic inverse with alpha < 0} with $g$ replaced by $g_0$.
	Denoting by $h_0 := g_0^{-1}$, we can verify that \eqref{eq: inverse between h and f} is valid with $h$ replaced by $h_0$.
	According to Lemma \ref{lem: regularly variation and inverse}, $h$ and $h_0$ are asymptotically equivalent at $\infty$.
	Hence, so are $g$ and $g_0$.
\end{proof}

\begin{lem}
  \label{lem:regularly_variation_and_integration}
  Let $(E, \mathscr E)$ be a measurable space and $\rho$ a finite non-degenerate measure on $(E, \mathscr E)$.
  Let $\alpha$ be a bounded measurable function on $E$ with 
  \begin{align}
    \alpha_0
    := \operatorname*{ess\,inf}_{\rho(dx)} \alpha(x)
    := \sup\{r:\rho(x:\alpha(x) < r) = 0\} \in \mathbb R.
  \end{align}
	Then $\big(\int_E t^{\alpha(x)} \rho(dx)\big)_{t\in (0,\infty)} \in \mathcal R^0_{\alpha_0}$.
	Further, if $\rho \{x:\alpha(x) = \alpha_0\}>0$, then
  \begin{align}
    \int_E t^{\alpha(x)} \rho(dx)
    \stackrel[t\to 0]{}{\sim}  \rho\{x:\alpha(x) = \alpha_0\} t^{\alpha_0}.
  \end{align}
\end{lem}

\begin{proof}
	If $u \in (0,1]$, then we have
  \begin{align}
    \frac{\int_E u^{\alpha(x)} t^{\alpha(x)} \rho(dx)}{\int_E t^{\alpha(x)} \rho(dx)}
    \leq \frac{\int_E u^{\alpha_0} t^{\alpha(x)} \rho(dx)}{\int_E t^{\alpha(x)} \rho(dx)}
    = u^{\alpha_0},
    \quad t\in (0,\infty).
  \end{align}
	This implies that
  \begin{align}
    \limsup_{(0,\infty) \ni t\to 0}\frac{\int_E u^{\alpha(x)} t^{\alpha(x)} \rho(dx)}{\int_E t^{\alpha(x)} \rho(dx)}	
    \leq u^{\alpha_0}.
  \end{align}
	Also, for any $\epsilon \in (0,\infty)$, we have
  \begin{align}
    &\frac{\int_E u^{\alpha(x)} t^{\alpha(x)} \rho(dx)}{\int_E t^{\alpha(x)} \rho(dx)}
      \geq \frac{ \int_{ \alpha(x) \leq  \alpha_0 + \epsilon } u^{ \alpha(x) } t^{ \alpha(x)} \rho(dx) } { \int_E t^{ \alpha(x) } \rho(dx) }
    \\&\quad \geq u^{ \alpha_0 + \epsilon} \frac{ \int_{ \alpha(x) \leq \alpha_0 + \epsilon } t^{ \alpha(x)} \rho(dx) } { \int_{ \alpha(x) \leq \alpha_0 + \epsilon}t^{\alpha(x)} \rho(dx)+ \int_{\alpha(x) > \alpha_0 + \epsilon} t^{\alpha(x)} \rho(dx)}
    \\&\quad = u^{\alpha_0 + \epsilon} \frac{1}{1+ \frac{\int_{\alpha(x) > \alpha_0 + \epsilon}t^{\alpha(x) - (\alpha_0 + \epsilon)} \rho(dx)}{\int_{\alpha(x) \leq \alpha_0 + \epsilon}t^{\alpha(x)- (\alpha_0 + \epsilon)} \rho(dx)}},
    \quad t\in (0, \infty),
    \\&\quad \xrightarrow[(0,\infty) \ni t\to 0]{}
    u^{\alpha_0 + \epsilon},
  \end{align}
	where the last convergence is due to the monotone convergence theorem.
	Therefore
  \begin{align}
    \liminf_{(0,\infty) \ni t\to 0}\frac{\int_E u^{\alpha(x)} t^{\alpha(x)} \rho(dx)}{\int_E t^{\alpha(x)} \rho(dx)}
    \geq u^{\alpha_0}.
  \end{align}
	Summarizing the above,  we get
  \begin{align}
    \lim_{(0,\infty) \ni t\to 0}\frac{\int_E u^{\alpha(x)} t^{\alpha(x)} \rho(dx)}{\int_E t^{\alpha(x)} \rho(dx)}
    = u^{\alpha_0},	
    \quad u \in (0,1].
  \end{align}
	If $u \in (1,\infty)$, taking $f(x, t):= t^{\alpha(x)}$, from what we have proved, we also have that
  \begin{align}
    \lim_{(0,\infty)\ni t\to 0}\frac{\int_E f(x,u t) \rho(dx)}{\int_E f(x, t) \rho(dx)}
    = \lim_{(0,\infty)\ni t\to 0}\frac{\int_E f(x,t) \rho(dx)}{\int_E f(x, u^{-1} t) \rho(dx)}
    = \big(( u^{-1})^{\alpha_0} \big)^{-1}
    = u^{\alpha_0}.
  \end{align}
	This proved the first part of the lemma.
	
	If further we have $\rho(x:\alpha(x) = \alpha_0)>0$, then by the monotone convergence theorem  it is easy to see that
  \begin{equation}
    \frac{\int_E t^{\alpha(x)} \rho(dx)}{t^{\alpha_0}}
    \xrightarrow[(0,\infty)\ni t\to 0]{} \rho(x:\alpha(x) = \alpha_0)\in (0,\infty).
    \qedhere
  \end{equation}
\end{proof}

\subsection{Superprocesses} 
\label{sec: Superprocesses}
In this subsection, we recall some known results on the $(\xi, \psi)$-superprocess $\{X; \mathbf P\}$.
It is known, see \cite[Theorem 2.23]{Li2011Measure-valued} for example, that \eqref{eq:FKPP_in_definition} can be written as
\begin{align}\label{eq:mean-fkpp}
	V_t f(x) + \int_0^t P^\beta_{t-r} \psi_0(x,V_r f) dr
	= P^\beta_t f(x),
	\quad f \in \mathscr B^+_b(E), t \geq 0,x \in E,
\end{align}
where
\begin{align}
	\psi_0(x,z)
	:= \psi(x,z) + \beta(x)z,
	\quad x \in E,z \geq 0.
\end{align}
Suppose that Assumptions \ref{asp: 1}--\ref{asp: 2} hold.
Since $\phi^*$ is the principal eigenfunction of the semigroup $(P_t^{\beta*})_{t\geq 0}$, we have
\begin{align}
  \langle P^\beta_{t} f, \phi^* \rangle_m
  = \langle f, P^{\beta*}_{t}\phi^* \rangle_m
  = \langle f, \phi^* \rangle_m,
  \quad f\in \mathscr B^+_b(E), t\geq 0.
\end{align}
Therefore, integrating both sides of \eqref{eq:mean-fkpp} with respect to the measure $\phi^*dm$, we get that
\begin{align}
  \langle V_tf,\phi^*\rangle_m + \int_0^t \langle \psi_0(\cdot ,V_r f) , \phi^*\rangle_mdr
  = \langle f,\phi^*\rangle_m,
  \quad t\geq 0, f\in \mathscr B^+_b(E).
\end{align}
This can be rearranged as
\begin{align}
  \label{eq:langleVtfphiranglem_equation}
  \langle V_tf,\phi^*\rangle_m + \int_s^t \langle \psi_0(\cdot ,V_r f) , \phi^*\rangle_mdr
  = \langle V_sf,\phi^*\rangle_m,
  \quad t\geq s\geq 0, f\in \mathscr B^+_b(E).
\end{align}

Let $\mathbb W$ be the collection of all $\mathcal M^1_E$-valued c\`{a}dl\`{a}g paths on $[0,\infty)$.
We refer to $\mathbb W$ as the \emph{canonical space of $(X_t)_{t\geq 0}$}.
In fact, $(X_t)$ can be viewed as a $\mathbb W$-valued random variable.
We denote the \emph{coordinate process of $\mathbb W$} by $(W_t)_{t\geq 0}$.

We say that $(X_t)_{t\geq 0}$ is \emph{non-persistent} if $\mathbf P_{\delta_x}(\|X_t\|= 0) > 0$ for all $x\in E$ and $t> 0$.
Suppose that $(X_t)_{t\geq 0}$ is non-persistent, then according to \cite[Section 8.4]{Li2011Measure-valued}, there is a family of measures $(\mathbb N_x)_{x\in E}$ on $\mathbb W$ such that
\begin{itemize}
\item
	for each $x\in E$, $\mathbb N_x (\forall t > 0, \|W_t\|=0) =0$;
\item
	for each $x\in E$, $\mathbb N_x(\|W_0 \|\neq 0) = 0$;
\item
	for any $\mu \in \mathcal M_E^1$, if $\mathcal N$ is a Poisson random measure on $\mathbb W$ with intensity $\mathbb N_\mu(\cdot):= \int_E \mathbb N_x(\cdot )\mu(dx)$, then the superprocess $\{X;\mathbf P_\mu\}$ can be realized by $\widetilde X_0 := \mu$ and $\widetilde X_t(\cdot) := \mathcal N[W_t(\cdot)],t>0$.
\end{itemize}
We refer to $(\mathbb N_x)_{x\in E}$ as the \emph{Kuznetsov measures} of $X$.
For the existence and further properties of such measures, we refer our readers to \cite{Li2011Measure-valued}.

From  Campbell's formula, see  the proof of \cite[Theorem 2.7]{Kyprianou2014Fluctuations} for example, we have
\begin{align} 
  \label{eq: equation for N measure}
	- \log \mathbf P_\mu [e^{-X_t(f)}]
	= \mathbb N_\mu[ 1-e^{- W_t(f)}],
	\quad \mu \in \mathcal M_E^1, t>0, f\in \mathscr B_b^+(E).
\end{align}
For each $x\in E$ and $t\geq 0$, taking $\mu = \delta_x$ and $f = \lambda \mathbf 1_E$ with $\lambda > 0$ in the above equation, and letting $\lambda \to \infty$, we get
\begin{align} 
  \label{eq: definition of v_t(x)}
	v_t(x)
	:= \lim_{\lambda\to \infty} V_t(\lambda\mathbf 1_E)(x)
	= -\log \mathbf P_{\delta_x} (\|X_t\|=0)
	= \mathbb N_x(\|W_t\|\neq 0).
\end{align}
For each $\mu \in \mathcal M_E^1$ and $t > 0$, by \eqref{eq: equation for N measure}, \eqref{eq: definition of v_t(x)} and the monotone convergence theorem, we have
\begin{align}
  \mathbb N_\mu(\|W_t\|\neq 0)
	&= -\log \mathbf P_{\mu} (\|X_t\|=0)
   = \lim_{\lambda \to \infty} (- \log \mathbf P_\mu [e^{-\lambda X_t(\mathbf 1_E)}])
	\\\label{eq: equation for mu v-t}
	&= \lim_{\lambda \to \infty} \langle \mu, V_t(\lambda \mathbf 1_E)\rangle
   = \mu(v_t).
\end{align}
It is also known that for any $f\in\mathscr B_b^+(E)$,
\begin{align}
\label{eq: mean of kuz measure}
  \mathbb N_{\mu}[W_t(f)]
  =\mathbf P_{\mu}[X_t(f)]=\mu(P^\beta_tf),
  \quad t \geq 0.
\end{align}

\subsection{Spine decompositions}
\label{sec: Spine decompositions}
Let $(\Omega, \mathscr F)$ be a measurable space with a $\sigma$-finite measure $\mu$.
For any $F\in \mathscr F$, we say \emph{$\mu$ can be size-biased by $F$} if $\mu(F< 0) = 0$ and $\mu(F) \in (0,\infty)$.
In this case, we define the \emph{$F$-transform of $\mu$} as the probability $\mu^F$ on $(\Omega, \mathscr F)$ such that
\begin{align}	
	d\mu^F
	= \frac{F}{\mu(F)}d \mu.
\end{align}

Let $\{X;\mathbf P\}$ be a non-persistent superprocess.
Let $\mu \in \mathcal M^1_E$ and $T>0$.
Suppose that $g\in \mathscr B^+(E)$ satisfies that $\mu(P^\beta_Tg) \in (0,\infty)$.
Then, according to \eqref{eq: mean of kuz measure}, $\mathbf P_\mu$ (resp. $\mathbb N_\mu$) can be size-biased by $X_T(g)$ (resp. $W_T(g)$).
Denote by $\mathbf P_\mu^{X_T(g)}$ (resp. $\mathbb N^{W_T(g)}_\mu$) the $X_T(g)$-transform of $\mathbf P_\mu$ (resp. the $W_T(g)$-transform of $\mathbb N_\mu$).
The spine decomposition theorem characterizes the law of $\{(X_t)_{t\geq 0}; \mathbf P_\mu^{X_T(g)}\}$ in two steps.
The first step of the theorem says that $\{(X_t)_{t\geq 0}; \mathbf P_\mu^{X_T(g)}\}$ can be decomposed in law as the sum of two independent measure-valued processes:

\begin{thm}[Size-biased decomposition, \cite{RenSongSun2017Spine}]
  \label{thm: size-biased decomposition}
  \begin{align}
    \{(X_t)_{t\geq 0}; \mathbf P_\mu^{X_T(g)}\}
    \overset{f.d.d.}{=} \{(X_t+W_t)_{t\geq 0}; \mathbf P_\mu  \otimes \mathbb N^{W_T(g)}_\mu\}.
  \end{align}
\end{thm}
The second step of the spine decomposition theorem says that $\{(W_t)_{0\leq t\leq T}; \mathbb N^{W_T(g)}_\mu\}$ has a spine representation, which intuitively says that, under probability $\mathbb N_\mu^{W_T(g)}$, the measure-valued process $(W_t)_{0\leq t\leq T}$ can be decomposed as a measure-valued immigration process along the trajectory of a spine process in a Poissonian way.

More precisely, we say $\{(\xi_t)_{0\leq t\leq T}, \mathbf n_T,  (Y_t)_{ 0\leq t\leq T}; \dot {\mathbf P}^{(g,T)}_\mu\}$ is a \emph{spine representation of $\mathbb N^{W_T(g)}_\mu$}  if:
\begin{itemize}
\item
	The \emph{spine process} $\{(\xi_t)_{0\leq t\leq T}; \dot{\mathbf P}^{(g,T)}_\mu\}$ is a copy of $\{(\xi_t)_{0\leq t\leq T}; \Pi^{(g,T)}_{\mu}\}$, where $\Pi^{(g,T)}_{\mu}$ is the $g(\xi_T) \exp\{\int_0^T \beta(\xi_s)ds\}$-transform of the measure $\Pi_{\mu}(\cdot):=\int_{E}\mu(dx)\Pi_x(\cdot) $;
\item
	Given $\{(\xi_t)_{0\leq t\leq T}; \dot{\mathbf P}^{(g,T)}_\mu\}$, the \emph{immigration measure} $\{\mathbf n_T; \dot{\mathbf P}^{(g,T)}_\mu[\cdot |(\xi_t)_{0\leq t\leq T}]\}$ is a Poisson random measure on $[0,T] \times \mathbb W$ with intensity
  \begin{align}
    \mathbf m^\xi_T(ds,dw)
    := 2 \alpha(\xi_s) ds \cdot \mathbb N_{\xi_s}(dw) + ds \cdot \int_{(0,\infty)} y \mathbf P_{y\delta_{\xi_s}}(X\in dw) \pi(\xi_s,dy);
  \end{align}
\item
	$\{(Y_t)_{0\leq t\leq T}; \dot{\mathbf P}^{(g,T)}_\mu\}$ is an $\mathcal M^1_E$-valued process defined by
  \begin{align}
    Y_t
    := \int_{(0,t] \times \mathbb W} w_{t-s} \mathbf n_T(ds,dw),
    \quad 0 \leq t\leq T.
  \end{align}
\end{itemize}
\begin{thm}[Spine representation, \cite{RenSongSun2017Spine}]\label{thm: spine representation}
	Let $\{(Y_t)_{0\leq t\leq T}; \dot {\mathbf P}^{(g,T)}_\mu\}$ be the spine representation of $\mathbb N^{W_T(g)}_\mu$ defined above.
	Then we have
  \begin{align}
    \{(Y_t)_{0\leq t\leq T}; \dot{\mathbf P}^{(g,T)}_\mu\}
    \overset{f.d.d.}{=} \{(W_t)_{0\leq t\leq T}; \mathbb N_\mu^{W_T(g)}\}.
  \end{align}
\end{thm}

Notice that $\mathbf P^{X_T(g)}_\mu(X_0 = \mu) = 1$.
Also notice that $\mathbb N_\mu$ is not a probability measure, but after the size-biased transform, $\mathbb N^{W_T(g)}_\mu$ is a probability measure.
Since $\mathbb N_{\mu}(\|W_0\|\neq 0) = 0$, we have $\mathbb N_\mu^{W_T(g)}(\|W_0\|= 0) = 1$.
Similarly, $\Pi_{\mu}$ is not typically a probability measure, but after the size-biased transform, $\Pi_{\mu}^{(T,g)}$ is a probability measure.
We note that
\begin{align}
	\Pi_{\mu}^{(T,g)} [ f(\xi_0) ]
	&= \frac{1}{\mu(P^\beta_Tg)}\Pi_{\mu}\Big[g(\xi_T)
   \exp\Big\{\int_0^T \beta(\xi_s)ds \Big\} f(\xi_0) \Big]
	\\&= \frac{1}{\mu(P^\beta_T g)}
	\int_E (P^\beta_T g)(x) \cdot f(x)\mu(dx),
\end{align}
which says that
\begin{align}
  \label{eq: initial distribution of spine}
	\Pi_{\mu}^{(T,g)} (\xi_0 \in dx)
	= \frac{1}{\mu(P^\beta_T g)} (P^\beta_T g)(x)\mu(dx),
	\quad x\in E.
\end{align}

Now, suppose that $\{\xi; \Pi\}$ satisfies Assumption \ref{asp: 1}.
Recall that $\phi$ is the principal eigenfunction of the mean semigroup of $X$.
The classical spine decomposition theorem, see \cite{EckhoffKyprianouWinkel2015Spines, EnglanderKyprianou2004Local, LiuRenSong2009Llog} for example, considered the case $g = \phi$ only.
In this case, the family of probabilities $(\Pi_{\mu}^{(\phi,T)})_{T\geq 0}$ is consistent in the sense of Kolmogorov's extension theorem, that is, the process $\{(\xi_t)_{0\leq t\leq T}; \Pi_{\mu}^{(\phi,T)} \}$ can be realized as the restriction of some process, say $\{(\xi_t)_{t\geq 0}; \Pi_{\mu}^{(\phi)}\}$,	on the finite time interval $[0,T]$.
In fact, one can also check that this consistency property is satisfied by  $(\mathbf P_\mu^{X_T(\phi)} )_{T\geq 0}$, $(\mathbb N^{W_T(\phi)}_\mu)_{T\geq 0}$ and $(\dot {\mathbf P}^{(\phi,T)}_\mu)_{T\geq 0}$.
Therefore, the actual statement of the classical spine decomposition theorem is different from merely replacing $g$ with $\phi$ in Theorem \ref{thm: size-biased decomposition} and \ref{thm: spine representation}:
There is no need to restrict the corresponding processes on the finite time interval $[0,T]$.
Because of its theoretical importance, we state the classical spine decomposition theorem explicitly here:

\begin{cro}
	For each $\mu \in \mathcal M_E^\phi \cap \mathcal M_E^1$, we have
  \begin{align}
    \{(X_t)_{t\geq 0}; \mathbf P_\mu^{(\phi)}\}
    \overset{f.d.d.}{=} \{(X_t + W_t)_{t\geq 0}; \mathbf P_\mu \otimes \mathbb N^{(\phi)}_\mu\}.
  \end{align}
	Here, the probability $\mathbf P_\mu^{(\phi)}$ is Doob's $h$-transform of $\mathbf P_\mu$ whose restriction on the natural filtration $(\mathscr F_t^X)$ of the process $(X_t)_{t\geq 0}$ is
  \begin{align}
    d ( \mathbf P_\mu^{(\phi)}|_{\mathscr F_t^X})
    = \frac{X_t(\phi)}{ \mu(\phi)} d(\mathbf P_\mu|_{\mathscr F_t^X}),
    \quad t\geq 0;
  \end{align}
	and $\mathbb N_\mu^{(\phi)}$ is a probability measure on $\mathbb W$ whose restriction on the natural filtration $(\mathscr F_t^W)$ of the process $(W_t)_{t\geq 0}$ is
  \begin{align}
    d(\mathbb N_\mu^{(\phi)} |_{\mathscr F^W_t}  )
    = \frac{W_t(\phi)}{\mu(\phi)} d(\mathbb N_\mu |_{\mathscr F^W_t}  ),
    \quad t\geq 0.
  \end{align}
\end{cro}

Let $\mu \in \mathcal M^{(\phi)}_\mu$, we say $\{(\xi_t)_{t\geq 0}, \mathbf n, (Y_t)_{ t\geq 0}; \dot {\mathbf P}^{(\phi)}_\mu\}$ is a \emph{spine representation of $\mathbb N^{(\phi)}_\mu$} if:
\begin{itemize}
\item
	The \emph{spine process} $\{(\xi_t)_{t\geq 0}; \dot{\mathbf P}^{(\phi)}_\mu\}$ is a copy of $\{(\xi_t)_{t\geq 0}; \Pi^{(\phi)}_{\mu}\}$ where the probability $\Pi_{\mu}^{(\phi)}$ is Doob's $h$-transform of $\Pi_\mu$ whose restriction on the natural filtration $(\mathscr F_t^\xi)$ of the process $(\xi_t)_{t\geq 0}$ is
  \begin{align}
    d(\Pi_{\mu}^{(\phi)} |_{\mathscr F_t^\xi})
    = \frac{\phi(\xi_t)e^{\int_0^t \beta(\xi_s)ds}}{\mu(\phi)} d(\Pi_{\mu} |_{\mathscr F_t^\xi}),
    \quad t\geq 0;
  \end{align}
\item
	Conditioned on $\{(\xi_t)_{t\geq 0}; \dot{\mathbf P}^{(\phi)}_\mu\}$, \emph{the immigration measure} $\{\mathbf n; \dot{\mathbf P}^{(\phi)}_\mu[\cdot |(\xi_t)_{t\geq 0}]\}$ is a Poisson random measure on $[0,\infty ) \times \mathbb W$ with intensity
  \begin{align}
    \label{eq:meanMeasImmigr}
    \mathbf m^\xi(ds,dw)
    := 2 \alpha(\xi_s) ds \cdot \mathbb N_{\xi_s}(dw) + ds \cdot \int_{(0,\infty)} y \mathbf P_{y\delta_{\xi_s}}(X\in dw) \pi(\xi_s,dy);
  \end{align}
\item
	$\{(Y_t)_{t\geq 0}; \dot{\mathbf P}^{(\phi)}_\mu\}$ is an $\mathcal M^1_E$-valued process defined by
  \begin{align}
    \label{eq:defSpinImmigr}
    Y_t
    := \int_{(0,t] \times \mathbb W} w_{t-s} \mathbf n(ds,dw),
    \quad t\geq 0.
  \end{align}
\end{itemize}

\begin{cro}
	Let $\{(Y_t)_{t\geq 0}; \dot {\mathbf P}^{(\phi)}_\mu\}$ be the spine representation of $\mathbb N^{(\phi)}_\mu$ defined above.
	Then we have
  \begin{align}
    \{(Y_t)_{t\geq 0}; \dot{\mathbf P}^{(\phi)}_\mu\}
    \overset{f.d.d.}{=} \{(W_t)_{t\geq 0}; \mathbb N_\mu^{(\phi)}\}.
  \end{align}
\end{cro}

For the sake of generality, the spine decomposition theorems above are all stated with respect to a general initial configuration $\mu$.
If $\mu = \delta_x$ for some $x\in E$, then by \eqref{eq: initial distribution of spine}, we have $\Pi_{\delta_x}^{(T,g)} (\xi_0 = x) = 1$, so sometimes we write $\Pi_x^{(T,g)}$ for $\Pi_{\delta_x}^{(T,g)}$.
Similarly, we write $\Pi_x^{(\phi)}$ for $\Pi_{\delta_x}^{(\phi)}$.

\subsection{Ergodicity of the spine process}
\label{sec: Ergodicity}
In this subsection, we discuss the ergodicity of the spine process $\{(\xi_t)_{t\geq 0}; (\Pi^{(\phi)}_x)_{x\in E}\}$ under Assumptions \ref{asp: 1}--\ref{asp: 3}.
According to \cite{KimSong2008Intrinsic}, $\{\xi; \Pi^{(\phi)}_x\}$ is a time homogeneous Hunt process and its transition density with respect to the measure $m$ is
\begin{align}
	q_t(x,y) := \frac{\phi(y)}{\phi(x)} p^\beta_t(x,y),
	\quad x,y\in E, t>0.
\end{align}
Let $c_0>0$ and $c_1>0$ be the constants  in \eqref{eq:q(t,x,y)}, then we have
\begin{align}
  \label{eq: asymptotic for q_t(x,y)}
	\sup_{x\in E} \Big| \frac{q_t(x,y)}{\phi(y)\phi^*(y)} - 1\Big|
	\leq c_0 e^{-c_1 t},
	\quad t > 1.
\end{align}
This implies that the process $\{\xi; \Pi^{(\phi)}_x\}$ is ergodic.
One can easily get from \eqref{eq: asymptotic for q_t(x,y)} that $(\phi\phi^*)(x)m(dx)$ is the unique invariant probability measure of $\{\xi; \Pi^{(\phi)}_x\}$.
The following two lemmas are also simple consequences of \eqref{eq: asymptotic for q_t(x,y)}.
They will be needed in the proof of  Theorem \ref{thm: main theorem}(3).
\begin{lem}[{\cite[Lemma 5.1]{RenSongSun2017Spine}}]
  \label{lem: ergodicity of the underlying process}
  If $F\in \mathscr B_b(E\times [0,1]\times [0,\infty)$ is such that $F(y,u):= \lim_{t\to \infty} F(y,u,t)$ exists for each $y\in E$ and $u \in [0,1]$, then
  \begin{align}
    \int_0^1 F(\xi_{(1-u)t},u,t) du
    \xrightarrow[t\to \infty]{ L^2(\Pi_x^{(\phi)})}
    \int_0^1 \langle F(\cdot , u), \phi\phi^*\rangle_m du,
    \quad x\in E.
  \end{align}
\end{lem}
\begin{lem}
  \label{lem: Fatou-ergodic lemma for the uderlying process}
  Let $F\in \mathscr B^+_b(E\times [0,1]\times [0,\infty))$.
	Define $F(y,u):= \limsup_{t\to \infty} F(y,u,t)$ for each $y\in E$ and $u \in [0,1]$.
	Then, for each $x\in E$ and $p \geq 1$,
  \begin{align}
    \limsup_{ t \to \infty}  \Big\| \int_0^1 F(\xi_{(1-u) t },u,t) du  \Big\|_{\Pi_x^{(\phi)};L^p}
    \leq \int_0^1 \langle F(\cdot, u), \phi \phi^*\rangle_m du,
    \quad x\in E.
  \end{align}
\end{lem}

\begin{proof}
	For each $(y,u,t)\in E\times [0,1]\times [0,\infty)$, define $\bar F(y,u,t) := \sup_{s:s\geq t} F(y,u,s)$.
  Then $\bar F\in \mathscr B_b(E\times [0,1]\times [0,\infty))$ and
  \begin{align}
    F(x,u)
    = \lim_{t\to \infty} \bar F(x,u,t),
    \quad x\in E, u\in [0,1].
  \end{align}
  From Lemma \ref{lem: ergodicity of the underlying process}, we know that
  \begin{align}
    \int_0^1 \bar F(\xi_{(1-u)t},u,t) du
    \xrightarrow[t\to \infty]{L^2(\Pi_x^{(\phi)})}
    \int_0^1 \langle F(\cdot , u), \phi\phi^*\rangle_m du,
    \quad x\in E,
  \end{align}
  which implies convergence in probability.
  The bounded convergence theorem then gives that, for each $p \geq 1$,
  \begin{align}
    \int_0^1 \bar F(\xi_{(1-u)t},u,t) du
    \xrightarrow[t\to \infty]{L^p(\Pi_x^{(\phi)})}
    \int_0^1 \langle F(\cdot , u), \phi\phi^*\rangle_m du,
    \quad x\in E.
  \end{align}
  Finally, noting that $0\leq F \leq \bar F$, we get
  \begin{align}
    \limsup_{ t \to \infty}  \Big\| \int_0^1 F(\xi_{(1-u) t },u,t) du  \Big\|_{\Pi_x^{(\phi)};L^p}
    &\leq \limsup_{ t \to \infty}  \Big\| \int_0^1 \bar F(\xi_{(1-u) t },u,t) du  \Big\|_{\Pi_x^{(\phi)};L^p}
    \\& = \int_0^1 \langle F(\cdot, u), \phi \phi^*\rangle_m du,
    \quad x\in E. \qedhere
  \end{align}
\end{proof}

\section{Proofs}
\subsection{Proof of Theorem \ref{thm: main theorem}(1)}
\label{sec: proof of result 1}
Let $\{X; \mathbf P\}$ be a $(\xi, \psi)$-superprocess satisfying Assumptions \ref{asp: 1}--\ref{asp: 4}.
In this subsection, we will prove the following result stronger than non-persistency:

\begin{prop} 
  \label{prop: non-presistent}
	For each $t > 0$, $\inf_{x\in E} \mathbf P_{\delta_x}(\|X_t\|= 0) > 0$.
\end{prop}

\begin{proof}
  Recall that $\kappa_0 = \operatorname{ess\,inf}_{m(dx)} \kappa(x) $ and $\gamma_0 = \operatorname{ess\,inf}_{m(dx)} \gamma(x)$.
	For each $x\in E$, let $\widetilde \kappa(x) := \kappa(x) \mathbf 1_{\kappa(x)\geq \kappa_0} + \kappa_0 \mathbf 1_{\kappa(x) < \kappa_0}$ and $\widetilde \gamma(x) := \gamma(x) \mathbf 1_{\gamma(x)\geq \gamma_0} + \gamma_0 \mathbf 1_{\gamma(x) < \gamma_0}$.
  Then, we know that $m(\widetilde \kappa \neq \kappa) = 0$ and $m(\widetilde \gamma \neq \gamma) = 0$.
	Define $\widetilde \psi(x,z) := - \beta(x)z+ \widetilde \kappa(x)z^{\widetilde \gamma(x)}$ for each $x\in E$ and $z\geq 0$, then for each $z\geq 0$, $\widetilde \psi(\cdot, z) = \psi(\cdot , z),$ $m$-almost everywhere.
  
	If we replace $\psi$ with $\widetilde\psi$ in \eqref{eq:FKPP_in_definition}, the solution $V_tf(x)$ of equation \eqref{eq:FKPP_in_definition} is also the solution of
  \begin{align}
    V_t f(x) + \Pi_x \Big[  \int_0^{t\wedge \zeta} \widetilde \psi (\xi_s,V_{t-s} f) ds \Big]
    =\Pi_x \big[ f(\xi_t)\mathbf 1_{t<\zeta} \big].
  \end{align}
  So, we can consider $\{X; \mathbf P\}$ as a superprocess with branching mechanism $\widetilde \psi$.
  Define
  \begin{align}
    \widehat\psi(z)
    := - (\|\beta\|_\infty +\kappa_0 )z + \kappa_0 z^{\gamma_0},
    \quad z\geq 0.
  \end{align}
  Using the fact that $\gamma_0 > 1$ and $\kappa_0 > 0$, it is easy to verify that
  \begin{align}
    \inf_{x\in E}\widetilde \psi(x,z)
    \geq \widehat\psi(z),
    \quad z\geq 0;
    \quad \int_1^\infty \frac{1}{\widehat\psi(z)} dz
    < \infty;
    \quad \widehat \psi(+\infty) = +\infty.
  \end{align}
  Therefore $\widetilde \psi$ satisfies the condition of \cite[Lemma 2.3]{RenSongZhang2015Limit}.
  As a consequence, we have the desired result.
\end{proof}

\subsection{Proof of Theorem \ref{thm: main theorem}(2)}
\label{sec: proof of result 2}
\begin{proof}
  [Proof of Theorem \ref{thm: main theorem}(2)]
  Let $\{X; \mathbf P\}$ be a $(\xi, \psi)$-superprocess satisfying
  Assumptions \ref{asp: 1}--\ref{asp: 4}.
  From Proposition \ref{prop: non-presistent}, we know that 
  \begin{align}
    \label{eq:extinction_probability_is_not_zero}
    \inf_{x\in E} \mathbf P_{\delta_x}(\|X_t\| = 0)
    > 0,
  \end{align}
  which implies that $\{X; \mathbf P\}$ is non-persistent.
  According to \eqref{eq:mean_formula}, Assumption \ref{asp: 2} and the fact that $\phi$ is the principal eigenfunction of the semigroup $(P_t^\beta)_{t\geq 0}$, we have $\mathbf P_{\delta_x}[X_t(\phi)] = P_t^\beta \phi(x) = e^{\lambda t} \phi(x)= \phi(x)>0$.
  Therefore,
  \begin{align}
    \label{eq:extinction_probability_is_not_one}
    \mathbf P_{\delta_x}(\|X_t\|= 0)<1,
    \quad t>0, x \in E.
  \end{align}
  From \eqref{eq:extinction_probability_is_not_zero}, \eqref{eq:extinction_probability_is_not_one} and \eqref{eq: definition of v_t(x)}, we have that $v_t \in \mathscr B^{++}_b(E)$ for each $t > 0$.
  
  According to \eqref{eq: definition of v_t(x)} and \eqref{eq:mean-fkpp}, by monotonicity, we see that $(v_t)_{t > 0}$ satisfies the equation
  \begin{align}
    v_{s+t}(x) + \int_0^t P^\beta_{t-r} \psi_0(x,v_{s+r}) dr
    = P^\beta_t v_s(x)
    \in [0,\infty),
    \quad s>0, t \geq 0,x \in E.
  \end{align}
  Notice that, under Assumption \ref{asp: 1}, according to \eqref{eq: p-t-beta is comparable to phi phi-star}, $d\nu:= \phi^* dm$ defines a finite measure on $E$.
  Therefore, $\langle v_t, \phi^*\rangle_m < \infty$ for each $t>0$.
  According to \eqref{eq:langleVtfphiranglem_equation}, \eqref{eq: definition of v_t(x)} and the monotone convergence theorem, $(v_t)_{t> 0}$ also satisfies the equation
  \begin{align}
    \label{eq: equation of <vt,phi>}
    \langle v_t,\phi^*\rangle_m + \int_s^t \langle \psi_0(\cdot ,v_t) , \phi^*\rangle_m dr
    = \langle v_s,\phi^*\rangle_m
    \in [0,\infty),
    \quad s, t > 0.
  \end{align}
  One of the consequences of this equation is that, see \cite[Lemma 5.2]{RenSongSun2017Spine} for example,
  \begin{align}
    \label{eq: uniform converges to 0}
    \|\phi^{-1}v_t\|_{\infty} \xrightarrow[t\to \infty]{} 0.
  \end{align}
  However, to prove Theorem \ref{thm: main theorem}(2), we need to consider the speed of this convergence.
  This is answered in the following two propositions whose proofs are postponed after this proof.
  The first proposition says that $(\phi^{-1}v_t)(x)$ will converge to $0$ with the same speed as $\langle v_t,\phi^*\rangle_m$, uniformly in $x\in E$:
  
  \begin{prop}
    \label{prop: convergence in a same speed}
    $(\phi^{-1}v_t)(x) \stackrel[t\to\infty]{x\in E}{\sim} \langle v_t,\phi^*\rangle_m$.
  \end{prop}
  The second proposition characterizes this speed:
  \begin{prop}
    \label{prop: regularly varying of vt-phi-star}
    $(\langle v_t,\phi^*\rangle_m)_{t> 0} $ is regularly varying at $\infty$ with index $-\frac{1}{\gamma_0-1}$.
    Furthermore, if $m(x: \gamma (x)= \gamma_0)>0$, then
    \begin{align}
      \langle v_t,\phi^*\rangle_m
      \stackrel[t\to \infty]{}{\sim} \big(C_X(\gamma_0-1) t \big)^{-\frac{1}{\gamma_0 - 1}},
    \end{align}
    where $C_X:= \langle \mathbf 1_{\gamma= \gamma_0} \kappa \phi^{\gamma_0}, \phi^* \rangle_m $.
  \end{prop}
  It follows from \eqref{eq: equation for mu v-t} and \eqref{eq: uniform converges to 0} that
  \begin{align}
    - \log \mathbf P_\mu(\|X_t\| = 0)
    = \mu(v_t)
    \leq \mu(\phi) \| \phi^{-1} v_t\|_{\infty}
    \xrightarrow[t\to \infty]{} 0.
  \end{align}
  Therefore, since $- \log (1-x) \to 0$ implies $x \to 0$, we have $\mathbf P_\mu(\|X_t\| \neq 0) \xrightarrow[t\to \infty]{} 0$.

  It follows from the fact that $x \stackrel[x\to 0]{}{\sim} - \log(1-x)$, \eqref{eq: equation for mu v-t}, Lemma \ref{lem: asymptotic equivalent of integration} and Proposition \ref{prop: convergence in a same speed} that
  \begin{align}
    \mathbf P_\mu(\|X_t\| \neq 0)
    \stackrel[t\to \infty]{}{\sim} - \log \mathbf P_\mu(\|X_t\| = 0)
    = \mu(\phi \phi^{-1}v_t)
    \stackrel[t\to\infty]{}{\sim} \mu(\phi) \langle v_t, \phi^*\rangle_m.
  \end{align}
  Then the desired result follows immediately from Proposition \ref{prop: regularly varying of vt-phi-star}.
\end{proof}

\begin{proof}
  [Proof of Proposition \ref{prop: convergence in a same speed}]
	We use an argument similar to that used in \cite{RenSongSun2017Spine} for critical superprocesses with finite 2nd moment.
  We only need to prove that there exists a map $t\mapsto a_t > 0$ such that
  \begin{align}
    \label{eq:k1}
    \sup_{x\in E} \Big| \frac{ ( \phi^{-1} v_t )(x)}{ a_t} - 1 \Big|
    \xrightarrow[t\to \infty]{} 0.
  \end{align}
  In fact, once this is proved, we will have that
  \begin{align}
    \label{eq:k2}
    \Big |\frac {\langle v_t, \phi^*\rangle_m} {a_t } - 1 \Big |
    & \leq \int \Big | \frac{(\phi^{-1}v_t)(x)}{ a_t } - 1 \Big| \phi \phi^*(x) m(dx)
    \\ & \leq \sup_{x\in E}\Big|\frac{(\phi^{-1}v_t)(x)}{ a_t }-1 \Big|
         \xrightarrow[t\to\infty]{} 0.
  \end{align}
  Then, by \eqref{eq:k1}, \eqref{eq:k2} and the property of uniform convergence, we will get the desired result:
  \begin{align}
    \sup_{x\in E}\Big|\frac{(\phi^{-1}v_t)(x)}{\langle v_{t},\phi^* \rangle_m}-1 \Big|
    \xrightarrow[t\to\infty]{} 0.
  \end{align}
	
  For each $\mu\in\mathcal M^\phi_E$, denote by $\{(Y_t), (\xi_t),\mathbf n; \dot {\mathbf P}^{(\phi)}_\mu\}$ the spine representation of $\mathbb N^{(\phi)}_\mu$.
	According to \eqref	{eq: equation for mu v-t}, \eqref{eq: mean of kuz measure} and Theorem \ref{thm: spine representation}, we have that for each $t>0$,
  \begin{align}
    \label{eq:vt-and-Y}
    \langle \mu,\phi \rangle \dot {\mathbf P}^{(\phi)}_\mu [Y_t(\phi)^{-1}]
    = \mathbb N_\mu[W_t(\phi)] \mathbb N^{W_t(\phi)}_\mu [W_t(\phi)^{-1}]
    = \mathbb N_\mu(W_t(\phi) > 0)
    = \mu(v_t).
  \end{align}
  Taking $\mu = \delta_x$ in \eqref{eq:vt-and-Y}, we get $(\phi^{-1}v_t)(x) =\dot{\mathbf P}_{\delta_x}^{(\phi)}[Y_t(\phi)^{-1}]$.
  Recall that $d\nu = \phi^* dm$.
  Taking $\mu = \nu$ in \eqref{eq:vt-and-Y}, we get $\langle v_t, \phi^*\rangle_m = \dot {\mathbf P}_{\nu}^{(\phi)} [Y_t(\phi)^{-1}]$.
  
  In order to construct an $(a_t)_{t\geq 0}$ satisfying \eqref{eq:k1}, we consider a decomposition of the immigration process $(Y_t)_{t\geq 0}$.
  For any $t>0$ and any $G\in \mathscr B((0,t])$, define
  \begin{align}
    Y^G_t
    := \int_{G\times \mathbb W} w_{t-s} \mathbf n(ds,dw).
  \end{align}
  Then for any $0 < t_0 < t$, we can decompose $Y_t$ into
  \begin{align}
    Y_t
    = Y^{(0,t_0]}_t + Y^{(t_0,t]}_t.
  \end{align}
  Using this decomposition, for each $0<t_0<t<\infty$ and $x\in E$, we have
  \begin{align}
    \label{eq: starting point of phi-1v_t(x)}
    \dot{\mathbf P}_{\delta_x}^{(\phi)}[Y_t(\phi)^{-1}]
    = \dot {\mathbf P}_\nu^{(\phi)} [Y^{(t_0,t]}_t(\phi)^{-1}] + \epsilon_x^1(t_0,t) +\epsilon_x^2(t_0,t),
  \end{align}
  where
  \begin{align}
    \epsilon_x^1(t_0,t)
    &:= \dot {\mathbf P}_{\delta_x}^{(\phi)} [Y^{(t_0,t]}_t(\phi)^{-1}] - \dot {\mathbf P}_\nu^{(\phi)} [Y^{(t_0,t]}_t(\phi)^{-1}];
    \\\epsilon_x^2(t_0,t)
    &:= \dot{\mathbf P}_{\delta_x}^{(\phi)}[Y_t(\phi)^{-1} - Y^{(t_0,t]}_t(\phi)^{-1}].
  \end{align}

  By the construction of the spine representation $\{(Y_t), (\xi_t),\mathbf n; \dot {\mathbf P}^{(\phi)}_\mu\}$ and its Markov property, we have that
  \begin{align}
    \label{eq: some equations for PY-1-1}
    &\dot{\mathbf P}^{(\phi)} [Y_t^{(t_0,t]}(\phi)^{-1}|\mathscr F^\xi_{t_0}]
      = \dot{\mathbf P}_{\delta_{\xi_{t_0}}}^{(\phi)}  [Y_{t-t_0}(\phi)^{-1}]
      = (\phi^{-1}v_{t-t_0})(\xi_{t_0});
    \\ \label{eq: some equations for PY-1-2}
    &\dot{\mathbf P}_\nu^{(\phi)}[Y_t^{(t_0,t]}(\phi)^{-1}]
      = \Pi_{\nu}^{(\phi)}[(\phi^{-1}v_{t-t_0})(\xi_{t_0}) ]
      = \langle v_{t-t_0},\phi^* \rangle_m;
    \\ \label{eq: some equations for PY-1-3}
    &\dot{\mathbf P}_{\delta_x}^{(\phi)}[Y_t^{(t_0,t]}(\phi)^{-1}]
      = \Pi_x^{(\phi)}[(\phi^{-1}v_{t-t_0})(\xi_{t_0}) ]
      = \int_E  q_{t_0}(x,y)(\phi^{-1}v_{t-t_0})(y) m(dy).
  \end{align}
  
  We will show that both $\epsilon_x^1(t_0,t)$ and $\epsilon_x^2(t_0,t)$ are very small compared to $\dot {\mathbf P}_\nu^{(\phi)}[Y_t^{(t_0,t]}(\phi)^{-1}]$ provided $t_0$ and $t- t_0$ are large enough.
  This is done in the following two lemmas whose proofs are postponed after this proof.

  Let $c_0, c_1>0$ be the constants in \eqref{eq:q(t,x,y)}.
  \begin{lem}
    \label{lem:bound_for_epsilon1}
    For each $t > t_0 > 1$, we have that
    \begin{align}
      |\epsilon_x^1(t_0,t)|
      \leq c_0 e^{-c_1 t_0}\langle v_{t-t_0},\phi^* \rangle_m.
    \end{align}
  \end{lem}

  \begin{lem}
    \label{lem:upperbound_of_epsilon-2}
    For each $t_0 > 1$ and $t - t_0$ large enough, we have
    \begin{align}
      \label{eq:upperbound_of_epsilon-2}
      |\epsilon_x^2(t_0,t)|
      \leq t_0\|\kappa\gamma\phi^{\gamma - 1}\|_{\infty} \cdot \|\phi^{-1}v_{t-t_0}\|^{\gamma_0-1}_\infty (1+c_0 e^{-c_1 t_0}) \langle v_{t-t_0},\phi^* \rangle_m.
    \end{align}
  \end{lem}
  
  Now, for each $t_0 > 1$ and $t-t_0$ large enough, according to \eqref{eq:vt-and-Y},  \eqref{eq: starting point of phi-1v_t(x)}, \eqref{eq: some equations for PY-1-2}, Lemmas \ref{lem:bound_for_epsilon1} and \ref{lem:upperbound_of_epsilon-2}, we have
  \begin{align} 
    \label{vts-inequality}
    &\Big|\frac{(\phi^{-1}v_t)(x)}{\langle v_{t-t_0},\phi^* \rangle_m}-1 \Big|
      \leq \frac{|\epsilon_x^1(t_0,t)|}{\langle v_{t-t_0},\phi^* \rangle_m} + \frac{|\epsilon_x^2(t_0,t)|}{\langle v_{t-t_0},\phi^* \rangle_m}\\
    & \quad \leq c_0e^{-c_1 t_0} +t_0\|\kappa\gamma\phi^{\gamma - 1}\|_{\infty}
      \cdot \|\phi^{-1}v_{t-t_0}\|^{\gamma_0-1}_\infty (1+c_0 e^{-c_1 t_0}).
  \end{align}
  According to \eqref{eq: uniform converges to 0}, there exists a map $t\mapsto t_0(t)$ such that,
  \begin{align}
    t_0(t)
    \xrightarrow[t\to\infty]{} \infty;
    \quad t_0(t)\| \phi^{-1}v_{t-t_0(t)}\|^{\gamma_0 - 1}_\infty
    \xrightarrow[t\to\infty]{} 0.
  \end{align}
  Plugging this choice of $t_0(t)$ into \eqref{vts-inequality} and taking $t\to \infty$, we get the desired assertion \eqref{eq:k1} with $a_t := \langle v_{t-t_0(t)},\phi^* \rangle_m$.
\end{proof}

\begin{proof}
  [Proof of Lemma \ref{lem:bound_for_epsilon1}]
  Note that $c_0, c_1>0$ are the constants in \eqref{eq:q(t,x,y)}.
  Then for each $t > t_0 > 1$, we have that
  \begin{align}
    |\epsilon_x^1(t_0,t)|
    & = \big| \dot {\mathbf P}_{\delta_x}^{(\phi)} [Y^{(t_0,t]}_t(\phi)^{-1}] - \dot {\mathbf P}_\nu^{(\phi)} [Y^{(t_0,t]}_t(\phi)^{-1}] \big| \\
    & = \Big|  \int_E  q_{t_0}(x,y)(\phi^{-1}v_{t-t_0})(y) m(dy) - \langle v_{t-t_0},\phi^* \rangle_m \Big|\\
    & \leq \int_{y\in E} \big| q_{t_0}(x,y) - (\phi\phi^*)(y) \big| (\phi^{-1}v_{t-t_0})(y) m(dy)\\
    & \leq c_0 e^{-c_1 t_0}\langle v_{t-t_0},\phi^* \rangle_m.
      \qedhere
  \end{align}
\end{proof}

\begin{proof}
  [Proof of Lemma \ref{lem:upperbound_of_epsilon-2}]
  Using the Markov property of the spine process and properties of Poisson random measures, we have
\begin{align}
  \label{eq:epsilon-2}
	|\epsilon_x^2(t_0,t)|
	&= \big| \dot{\mathbf P}_{\delta_x}^{(\phi)}[Y_t(\phi)^{-1} - Y^{(t_0,t]}_t(\phi)^{-1}] \big|
	\\&= \dot{\mathbf P}_{\delta_x}^{(\phi)}[Y_t^{(0,t_0]}(\phi)\cdot Y_t(\phi)^{-1}\cdot Y^{(t_0,t]}_t(\phi)^{-1}]
	\\&\leq \dot{\mathbf P}_{\delta_x}^{(\phi)}[\mathbf 1_{Y_t^{(0,t_0]}(\phi)\neq 0}\cdot Y^{(t_0,t]}_t(\phi)^{-1}]
	\\&= \dot{\mathbf P}_{\delta_x}^{(\phi)} \big[\dot{\mathbf P}_{\delta_x}^{(\phi)}[\mathbf 1_{Y_t^{(0,t_0]}(\phi)\neq 0}|\mathscr F^\xi_{t_0}] \cdot \dot{\mathbf P}_{\delta_x}^{(\phi)} [ Y^{(t_0,t]}_t(\phi)^{-1}|\mathscr F^\xi_{t_0}] \big].
\end{align}

Since $\phi^{-1}v_s$ converges to $0$ uniformly when $s\to \infty$, we can choose  $s_0>0$ such that for any $s\geq s_0$, we have $\|\phi^{-1}v_s\|_{\infty} \leq 1$.
With this $s_0>0$, we claim that for each $t - t_0\geq s_0$ the following holds:
\begin{align}
  \label{eq:epsilon-2-one}
  \dot{\mathbf P}_{\delta_x}^{(\phi)}[\mathbf 1_{\| Y_t^{(0,t_0]}\|\neq 0}|\mathscr F^\xi_{t_0}]
  \leq t_0\|\kappa \gamma \phi^{\gamma - 1}\|_\infty \cdot \|\phi^{-1}v_{t-t_0}\|^{\gamma_0-1}_\infty.
\end{align}
We will verify this claim at the end of this proof.

On the other hand, according to \eqref{eq: asymptotic for q_t(x,y)} and \eqref{eq: some equations for PY-1-3}, we know that
\begin{align}
  \label{eq:epsilon-2-final}
	\dot{\mathbf P}_{\delta_x}^{(\phi)}[ Y^{(t_0,t]}_t(\phi)^{-1}]
	\leq (1+c_0 e^{-c_1 t_0}) \langle v_{t-t_0},\phi^* \rangle_m.
\end{align}
Therefore, from \eqref{eq:epsilon-2}, \eqref{eq:epsilon-2-one} and \eqref{eq:epsilon-2-final}, we get that
\begin{align}
  &|\epsilon_x^2(t_0,t)|
    \leq t_0 \|\kappa \gamma \phi^{\gamma - 1}\|_\infty \|\phi^{-1} v_{t-t_0}\|_\infty^{\gamma_0 - 1} \cdot \dot{\mathbf P}_{\delta_x}^{(\phi)} \big[ Y_t^{(t_0, t]}(\phi)^{-1} \big]
  \\ & \leq t_0 \|\kappa \gamma \phi^{\gamma - 1}\|_\infty \|\phi^{-1} v_{t-t_0}\|_\infty^{\gamma_0 - 1} \cdot (1+ c_0 e^{-c_1 t_0})\langle v_{t-t_0}, \phi^* \rangle_m,
\end{align}
as desired.

We now verify the claim \eqref{eq:epsilon-2-one}.
Note that, if $t-s > t-t_0 \geq s_0$, using the fact that $v_t$ is non-increasing in $t$, we get
\begin{align}
	\kappa(x)\gamma(x) v_{t-s}(x)^{\gamma(x)-1}
	\leq \|\kappa \gamma \phi^{\gamma - 1}\|_\infty \cdot \|\phi^{-1} v_{t-s}\|^{\gamma_0-1}_\infty
	\leq \|\kappa\gamma\phi^{\gamma - 1}\|_\infty \cdot \|\phi^{-1}v_{t-t_0}\|^{\gamma_0-1}_\infty.
\end{align}
Therefore, using Campbell's formula, \eqref{eq: definition of Gamma function} and the fact that $e^{-x} \geq 1-x$, we have,  for $t-t_0 \geq s_0$,
\begin{align}
	&\dot{\mathbf P}_{\delta_x}^{(\phi)}[\mathbf 1_{\| Y_t^{(0,t_0]}\|\neq 0}|\mathscr F^\xi_{t_0}]
   \leq - \log \big( 1- \dot{\mathbf P}_{\delta_x}^{(\phi)}[\mathbf 1_{ \| Y_t^{(0,t_0]}\|\neq 0}|\mathscr F^\xi_{t_0}]\big)
	\\&\quad =  - \log \lim_{\lambda \to \infty}\dot{\mathbf P}_{\delta_x}^{(\phi)}[e^{- \lambda Y_t^{(0,t_0]}(\mathbf 1_E) }|\mathscr F^\xi_{t_0}]
	\\&\quad = -\log \lim_{\lambda \to \infty} \exp\Big\{- \int_{[0,t]\times \mathbb W} \big( 1-\exp\{- \mathbf 1_{s\leq t_0} w_{t-s}(\lambda \mathbf 1_E)\}  \big) \mathbf m^\xi(ds,dw)\Big\}
	\\&\quad = \int_{[0,t]\times \mathbb W}\mathbf 1_{s\leq t_0} \mathbf 1_{ \|w_{t-s}\| \neq 0} \mathbf m^\xi(ds,dw)
	= \int_0^{t_0} ds \int_{(0,\infty)} y\mathbf P_{y\delta_{\xi_s}}[\mathbf 1_{ \|X_{t-s}\| \neq 0}]\pi(\xi_s,dy)
  \\&\quad= \int_0^{t_0} ds \int_{(0,\infty)} y (1-e^{-yv_{t-s}(\xi_s)})  \frac{\kappa(\xi_s)dy}{\Gamma(-\gamma(\xi_s)) y^{1+\gamma(\xi_s)}}
	= \int_0^{t_0} \big( \kappa \gamma  v_{t-s}^{\gamma - 1} \big) (\xi_s)ds
	\\&\quad \leq  t_0\|\kappa \gamma \phi^{\gamma - 1}\|_\infty \cdot \|\phi^{-1}v_{t-t_0}\|^{\gamma_0-1}_\infty.
\end{align}
This ends the verification of the claim \eqref{eq:epsilon-2-one}, and thus also completes the proof of Lemma \ref{lem:upperbound_of_epsilon-2}.
\end{proof}

\begin{proof}
  [Proof of Proposition \ref{prop: regularly varying of vt-phi-star}]
	From \eqref{eq: equation of <vt,phi>} we know that $\langle v_t,\phi^* \rangle_m$ is continuous and strictly decreasing in $t \in (0,\infty)$.
	Since the superprocess $(X_t)_{t\geq 0}$ is right continuous in the weak topology with the null measure as an absorbing state, we have that, for each $\mu \in \mathcal M_E^1$, $\mathbf P_\mu (\|X_t\| \neq 0) \xrightarrow[t\to 0]{} 1$.
	Taking $\mu = \nu$, according to \eqref{eq: equation for mu v-t}, we have that $\langle v_{t}, \phi^*\rangle_m \xrightarrow[t\to 0]{} +\infty$.
	On the other hand, according to \eqref{eq: uniform converges to 0}, we have
  \begin{align}
    \langle v_{t}, \phi^*\rangle_m
    = \langle \phi^{-1} v_t \phi, \phi^*\rangle_m
    \leq \|\phi^{-1} v_t\|_\infty \cdot \langle \phi, \phi^* \rangle_m
    \xrightarrow[t\to \infty]{} 0.
  \end{align}
  Therefore, the map $t\mapsto \langle v_t,\phi^*  \rangle_m$ has an inverse on $(0,\infty)$ which is denoted by
  \begin{align}
    R: (0,\infty) \to (0,\infty).
  \end{align}
	
  Now, if we denote by
  \begin{align}
    \epsilon_{t}(x)
    : = \frac{v_t(x)}{\langle v_t, \phi^*\rangle_m \phi(x)} - 1,
    \quad t>0, x\in E,
  \end{align}
  then we have
  \begin{align}
    \label{eq: change variable using inverse}
    v_t(x)
    = \big(1+ \epsilon_{R(\langle v_t,\phi^* \rangle_m)}(x) \big )\langle v_t,\phi^* \rangle_m \phi(x),
    \quad t>0, x\in E.
  \end{align}
  Further, by Proposition \ref{prop: convergence in a same speed} and the fact that $R(u)\xrightarrow[u\to 0]{} \infty$, we have
  \begin{align}
    \label{eq: epsilon R converges to 0}
    \sup_{x\in E}|\epsilon_{R(u)}(x)|
    \xrightarrow[u\to 0]{} 0.
  \end{align}
  
	Now, by \eqref{eq: equation of <vt,phi>}, we have
  \begin{align}
    \frac{d \langle v_r, \phi^* \rangle_m}{dr}
    = - \langle \psi_0(\cdot ,v_r) ,\phi^*\rangle_m
    > 0
    \quad a.e..
  \end{align}
	Therefore,
  \begin{align}
    s-t
    & = \int_t^s dr
      = \int_s^t \langle \psi_0(\cdot ,v_r), \phi^*\rangle _m^{-1} d\langle v_r ,\phi^* \rangle_m
    \\ & \overset{\text{by \eqref{eq: change variable using inverse} }}{=} \int_s^t \big\langle \psi_0\big( \cdot ,(1+ \epsilon_{R(\langle v_r,\phi^* \rangle_m)})\langle v_r,\phi^*\rangle_m \phi \big), \phi^* \big\rangle_m^{-1} d\langle v_r ,\phi^* \rangle_m
    \\ & = \int_{\langle v_s,\phi^*\rangle}^{\langle v_t, \phi^* \rangle} \big\langle \psi_0 \big( \cdot ,( 1 + \epsilon_{R(u)} ) u \phi \big), \phi^* \big\rangle_m^{-1} du.
  \end{align}
	Letting $t\to 0$, we get
  \begin{align}
    s
    = \int_{\langle v_s,\phi^*\rangle}^\infty \big\langle \psi_0 \big(\cdot ,( 1 + \epsilon_{R(u)} ) u \phi \big), \phi^* \big\rangle_m^{-1} du,
    \quad s\in (0,\infty).
  \end{align}
  Since $R$ is the inverse of $t\mapsto \langle v_t,\phi^*\rangle$, the above implies that
  \begin{align}
    \label{eq: integral equation for R}
    R(r)
    = \int_r^\infty \big\langle \psi_0 \big(\cdot ,( 1 + \epsilon_{R(u)} ) u \phi \big), \phi^* \big\rangle_m^{-1} du,
    \quad r\in (0,\infty).
  \end{align}
	
  We now check the regularly varying property of $R(r)$ at $r=0$.
  This can be done by considering the regularly varying property of $u\to \big\langle \psi_0 \big(\cdot ,( 1 + \epsilon_{R(u)} ) u \phi \big), \phi^* \big\rangle_m$ at 0.
  According to \eqref{eq: epsilon R converges to 0},  $1+ \epsilon_{R(u)}(x) \stackrel[u\to 0]{x\in E}{\sim} 1$.
  Since $\gamma(\cdot)$ is bounded, we have $\big(1+ \epsilon_{R(u)}(x)\big)^{\gamma(x)}\stackrel[u\to 0]{x\in E}{\sim} 1$.
  Therefore, from Lemma \ref{lem: asymptotic equivalent of integration}, we have that
  \begin{align}
    \label{eq: regularly part in the integration}
    & \big\langle \psi_0 \big(\cdot,( 1 + \epsilon_{R(u)} ) u \phi \big), \phi^* \big\rangle_m
    \\ & \quad = \big\langle \kappa (x)\big( 1 + \epsilon_{R(u)}(x)\big )^{\gamma(x)} u^{\gamma(x)} \phi(x)^{\gamma(x)} , \phi^*(x) \big\rangle_{m(dx)}
    \\ & \quad \stackrel[u\to 0]{}{\sim} \langle u^{\gamma(x)} , \kappa (x)\phi(x)^{\gamma(x)} \phi^*(x) \rangle_{m(dx)}.
  \end{align}
  According to Lemma \ref{lem:regularly_variation_and_integration}, and using the fact that $\kappa(x)\phi(x)^{\gamma(x)}$ is bounded and the measure $\phi^* dm$ is finite, we have that $\langle \psi_0\big(\cdot,(1+\epsilon_{R(u)})u\phi \big), \phi^* \rangle_m$ is regularly varying at $u = 0$ with index $\gamma_0$.
  Noticing that $-(\gamma_0 - 1) < 0$, according to Corollary \ref{cro: power law and ingetration} and \eqref{eq: integral equation for R}, $R$ is regularly varying at $0$ with index $-(\gamma_0 - 1)$.
  Therefore, from $R(\langle v_s, \phi^*\rangle_m) = s$ and Corollary \ref{cro: regularly varing and inverse with alpha < 0}, we have that $(\langle v_s, \phi^*\rangle_m)_{s\in (0,\infty)}$ is regularly varying at $\infty$ with index $-(\gamma_0 - 1)^{-1}$.
	
  Further, if $m\{x: \gamma(x) = \gamma_0\}> 0$, then according to Lemma \ref{lem:regularly_variation_and_integration} and \eqref{eq: regularly part in the integration}, we know that
  \begin{align}
    & \big\langle \psi_0 \big(\cdot,( 1 + \epsilon_{R(u)} ) u \phi \big), \phi^* \big\rangle_m
      \stackrel[u\to 0]{}{\sim}  \langle u^{\gamma(x)} , \kappa (x)\phi(x)^{\gamma(x)} \phi^*(x) \rangle_{m(dx)}
    \\ & \quad \stackrel[u\to 0]{}{\sim}  \langle \mathbf 1_{\gamma(x)= \gamma_0}, \kappa (x)\phi(x)^{\gamma_0} \phi^*(x) \rangle_{m(dx)} u^{\gamma_0}
         =: C_X u^{\gamma_0}.
  \end{align}
  Therefore, we have $\big\langle \psi_0 \big(\cdot,( 1 + \epsilon_{R(u)} ) u \phi \big), \phi^* \big\rangle_m^{-1} = u^{-\gamma_0} l(u)$, where $l(u)$ converges to the constant $C_X^{-1}$ when $u \to 0$.
  Now according to Corollary \ref{cro: power law and ingetration} and \eqref{eq: integral equation for R} we have that
  \begin{align}
    R(r)
    & = \int_r^\infty \big\langle \psi_0 \big(\cdot,( 1 + \epsilon_{R(u)} ) u \phi \big), \phi^* \big\rangle_m^{-1} du
      = \int_r^\infty u^{-\gamma_0} l(u) du
    \\ & = -\frac{1}{\gamma_0-1}\int_r^\infty l(u) du^{-(\gamma_0 - 1)}
    \\ & \stackrel[r\to 0]{}{\sim} C_X^{-1} (\gamma_0-1)^{-1} r^{-(\gamma_0 - 1)}.
  \end{align}
  Finally since $r\mapsto \langle v_r,\phi^*\rangle_m$ is the inverse of $r\mapsto R(r)$, from \cite[Proposition 1.5.15.]{BinghamGoldieTeugels1989Regular} and the above, we have
  \begin{align}
    \langle v_r,\phi^*\rangle_m
    \stackrel[r\to \infty]{}{\sim} \big(C_X (\gamma_0-1) r \big)^{-\frac{1}{\gamma_0 - 1}}.
    \qedhere
  \end{align}
\end{proof}

\subsection{Characterization of the one dimensional distribution}
\label{sec: conditional distribution}
Let $\{(X_t)_{t\geq 0}; \mathbf P\}$ be a $(\xi, \psi)$-superprocess satisfying Assumptions \ref{asp: 1}--\ref{asp: 4}.
Suppose  $m(x: \gamma(x) = \gamma_0)>0$.
Recall that we want to find a proper normalization $(\eta_t)_{t\geq 0}$ such that $\big\{\big(\eta_t X_t(f))_{t \geq 0}; \mathbf P_\mu(\cdot | \|X_t\| \neq 0\big)\big\}$ converges weakly to a non-degenerate distribution for a large class of functions $f$ and initial configurations $\mu$.
Our guess of $(\eta_t)$ is
\begin{align}
  \label{eq: definition of eta}
	\eta_t
	:= (C_X(\gamma_0 - 1) t)^{-\frac{1}{\gamma_0 - 1}},
	\quad t\geq 0,
\end{align}
because in this case
\begin{align}
 	\mathbf P_{\delta_x}[\eta_t X_t(f)|\|X_t\|\neq 0]
	= \frac{\mathbf P_{\delta_x}[\eta_t X_t(f) \mathbf 1_{\|X_t\|\neq 0}]} {\mathbf P_{\delta_x}(\|X_t\|\neq 0) }
	= \frac{\eta_t}{\mathbf P_{\delta_x}(\|X_t\|\neq 0)} P^\beta_t f(x)
	\stackrel[t\to \infty]{}{\sim}  \langle f,\phi^* \rangle_m.
\end{align}
Here we have used Theorem \ref{thm: main theorem}(2) and the fact that (see \eqref{eq:q(t,x,y)})
\begin{align}
 	P^\beta_t f(x)
 	= \int_E p_t^\beta(x,y)f(y)dy
 	\xrightarrow[t\to \infty]{}  \phi(x) \langle f,\phi^*\rangle_m.
\end{align}

From the point of view of Laplace transforms, the desired result that, for any $f\in \mathscr B^+_b(E)$ and $\mu\in \mathcal M_E^1$, $\big\{\big(\eta_t X_t(f)\big)_{t \geq 0}; \mathbf P_{\mu}(\cdot | \|X_t\| \neq 0)\big\}$ converge weakly to some probability distribution $F_f$ is equivalent to the following convergence:
\begin{align}
	\mathbf P_\mu[1- e^{-\theta \eta_t X_t(f)} | \|X_t\|\neq 0]
	= \frac{1- \exp\{- \mu\big(V_t(\theta \eta_t f)\big) \}}{\mathbf P_\mu(\|X_t\|\neq 0)}
	\xrightarrow[t\to \infty]{} \int_{[0,\infty)}(1-e^{-\theta u}) F_f(du).
\end{align}
According to Theorem \ref{thm: main theorem}(2) and $1-e^{-x} \stackrel[x\to 0]{}{\sim} x$, this is equivalent to
\begin{align}
  \label{eq: equivalent result}
	\frac{\mu\big( V_t(\theta \eta_t f) \big)}{ \eta_t}
	\quad \xrightarrow[t\to \infty]{} \mu(\phi) \int_{[0,\infty)}(1-e^{-\theta u})F_{f}(du).
\end{align}
Therefore, to establish the weak convergence of $\big\{\big(\eta_t X_t(f)\big)_{t \geq 0}; \mathbf P_{\mu}(\cdot | \|X_t\|\neq 0)\big\}$, one only needs to verify \eqref{eq: equivalent result}.

In order to investigate the convergence of $\mu\big( V_t(\theta \eta_t f) \big)/ \eta_t$, we need to investigate the properties of $\theta\to V_t(\theta f)$.
(Note that \eqref{eq:mean-fkpp} only gives the the dynamics of $t\to V_t(\theta f)$.)
This is done in the following proposition:

\begin{prop}
	For any $f\in \mathscr B^+_b(E),\theta \geq 0,x\in E$ and $T>0$, we have
  \begin{align}
    \label{eq: equation for Vt(theta f) for theta}
    V_T ( \theta f) ( x)
    = \phi( x) \int_0^\theta \Pi_x^{(\phi)} \Big[ \frac{ f(\xi_T) } { \phi(\xi_T) } \exp\Big\{ - \int_0^T \big( \kappa \gamma V_{T-s} (r f)^{ \gamma - 1} \big) ( \xi_s) ds\Big\} \Big] dr.
  \end{align}
\end{prop}

\begin{proof}
	It follows from Theorems \ref{thm: size-biased decomposition} and \ref{thm: spine representation} that
  \begin{align}
    \frac{ \mathbf P_{\delta_x}[X_T(f)e^{-\theta X_T(f)}] } {  \mathbf P_{\delta_x} [X_T(f)] }
    = \mathbf P_{\delta_x}^{X_T(f)} [e^{-\theta X_T(f)}]
    = \mathbf P_{\delta_x} [e^{-\theta X_T(f)}] \dot {\mathbf P}_x^{(T,f)}[e^{-\theta Y_T(f)}],
  \end{align}
  where $\{(\xi)_{0\le t\le T}, \mathbf n_T, (Y)_{0\le t\le T}; \dot {\mathbf P}^{(f,T)}_x\}$ is the spine representation of $\mathbb N^{W_T(f)}_x$ with $\mathbf m^\xi_T$ being the intensity of the immigration measure $\mathbf n_T$ conditioned on $\{(\xi)_{0\le t\le T}; \dot {\mathbf P}^{(f,T)}_x\}$.
  From this, we have
  \begin{align} 
    \label{eq: dynamic of theta on v_t theta reason 1}
    \frac{\partial}{\partial \theta}
    (-\log \mathbf P_{\delta_x}[e^{-\theta X_T(f)}])
    = \frac{\mathbf P_{\delta_x}[X_T(f)e^{-\theta X_T(f)}]}{\mathbf P_{\delta_x}[e^{-\theta X_T(f)}]}
    = P^\beta_T f(x) \dot {\mathbf P}_x^{(T,f)}[e^{-\theta Y_T(f)}].
  \end{align}
  On the other hand, if we write $F(s,w):= \mathbf 1_{s\leq T} w_{T-s}(f)$, then by Assumption \ref{asp: 4}, the spine representation, Campbell's formula and \eqref{eq: definition of Gamma function}, we have
  \begin{align}
    \label{eq: dynamic of theta on v_t theta reason 2}
    & -\log \dot {\mathbf P}^{(T,f)}_{x}[e^{-\theta \mathbf n_T(F)}|\mathbf m_T^\xi]
      = \mathbf m_T^\xi(1-e^{-\theta F})
    \\ & \quad = \int_0^T ds \int_{(0,\infty)} y \mathbf P_{y\delta_{\xi_s}}[1- e^{-\theta X_{T-s}(f)}] \pi(\xi_s,y)
    \\ & \quad = \int_0^T ds \cdot \kappa(\xi_s) \int_{(0,\infty)} \mathbf (1- e^{- y V_{T-s}(\theta f)(\xi_s)}) \frac{dy}{\Gamma(-\gamma(\xi_s)) y^{\gamma(\xi_s)}}
    \\ & \quad = \int_0^T \big(\kappa\gamma V_{T-s}(\theta f)^{\gamma-1}\big)(\xi_s) ds.
  \end{align}
  Note that, since $\mathbf n_T(F)= Y_T(f)$, we can derive from \eqref{eq: dynamic of theta on v_t theta reason 1} and \eqref{eq: dynamic of theta on v_t theta reason 2} that
  \begin{align}
    V_T(\theta f)(x)
    & = -\log \mathbf P_{\delta_x}[e^{-\theta X_T(f)}]
      = \int_0^\theta P^\beta_Tf(x)
      \dot {\mathbf P}_x^{(T,f)}[e^{-r  Y_T(f)}] dr
    \\ & =P^\beta_Tf(x)\int_0^\theta \Pi_x^{(T,f)} \Big[\exp\Big\{-\int_0^T \big(\kappa\gamma V_{T-s}(r f)^{\gamma-1}\big)(\xi_s)~ds\Big\}\Big] dr
    \\ &= \phi( x) \int_0^\theta \Pi_x^{(\phi)} \Big[ \frac{ f(\xi_T) } { \phi(\xi_T) } \exp\Big\{ - \int_0^T \big( \kappa \gamma V_{T-s} (r f)^{ \gamma - 1} \big) ( \xi_s) ds\Big\} \Big] dr,
  \end{align}
  as required.
\end{proof}

Replacing $\theta$ with $\theta \eta_T$ in \eqref{eq: equation for Vt(theta f) for theta}, we have
\begin{align}
  \label{eq: equation for normalized V_T}
  & \frac{V_T(\theta \eta_T f)(x)}{\eta_T}
	\\ & \quad= \phi(x) \frac{1}{\eta_T}\int_0^{\theta \eta_T} \Pi_x^{(\phi)} \Big[ \frac { f(\xi_T) } { \phi(\xi_T) } \exp\Big\{-\int_0^T \big(\kappa\gamma V_{T-s}(r f)^{\gamma-1}\big)(\xi_s) ds\Big\}\Big] dr
	\\ & \quad = \phi(x) \int_0^{\theta} \Pi_x^{(\phi)} \Big[ \frac { f(\xi_T) } { \phi(\xi_T) }  \exp\Big\{-\int_0^T \big(\kappa\gamma V_{T-s}(r \eta_T f)^{\gamma-1}\big)(\xi_s) ds\Big\}\Big] dr
	\\ & \quad = \phi(x)\int_0^{\theta} \Pi_x^{(\phi)} \Big[\frac{f(\xi_T)}{\phi(\xi_T)} \exp\Big\{-T\int_0^1 \big(\kappa\gamma V_{uT}(r \eta_T f)^{\gamma-1}\big)(\xi_{(1-u)T}) du\Big\}\Big] dr.
\end{align}

\subsection{Proof of Theorem \ref{thm: main theorem}(3)}
\label{sec: proof of result 3}
Consider the $(\xi, \psi)$-superprocess $\{X;\mathbf P\}$ which satisfies Assumptions \ref{asp: 1}--\ref{asp: 4}.
Suppose that $m( x:\gamma(x)=\gamma_0 )>0$.
Let $f \in \mathscr B^+(E)$ be  such that $ \langle f, \phi^* \rangle_m > 0$  and $c_f :=\| \phi^{-1}f \|_\infty < \infty$.

Without loss of generality, we assume that $\langle f, \phi^* \rangle_m = 1$.
We claim that, in order to prove Theorem \ref{thm: main theorem}(3), we only need to show that
\begin{align}
  \label{eq: we only need to proof this}
	g(t,\theta,x)
	:=\frac{V_t (\theta \eta_t f) (x)}{\eta_t \phi(x)}
	\xrightarrow[t\to \infty]{} G(\theta)
	:= \Big( \frac{1}{1+\theta^{-(\gamma_0 - 1)}} \Big)^{ \frac{1}{\gamma_0 - 1} },
	\quad x\in E, \theta \geq 0.
\end{align}
In fact, by \eqref{eq: equation for normalized V_T}, we have $\|V_t(\theta \eta_t f)/ \eta_t \|_\infty \leq \theta \|\phi\|_\infty \|\phi^{-1}f\|_\infty.$
Therefore, if \eqref{eq: we only need to proof this} is true, then by the bounded convergence theorem, for each $\mu \in \mathcal M^1_E$,
\begin{align}
	\frac{\mu\big(V_t (\theta \eta_t f)\big)}{\eta_t }
	\xrightarrow[t\to \infty]{} \mu(\phi)G(\theta),
\end{align}
which, by the discussion in Subsection \ref{sec: conditional distribution}, is equivalent to Theorem \ref{thm: main theorem}(3).

From Lemma \ref{lem: characterize the general Mittag-Leffler distribution}, we have that $G$ satisfies
\begin{align}
  \label{eq: equation for G}
	G(\theta)
	= \int_0^\theta e^{ - \frac{1} {\gamma_0 - 1} J_G(r)} dr,
	\quad \theta \geq 0,
\end{align}
where
\begin{align}
  \label{eq: definition for J_G}
	J_G(r):=
	\gamma_0 \int_0^1 G(ru^{\frac{1}{\gamma_0 - 1}}) ^{\gamma_0 - 1}\frac{du}{u},
	\quad r\geq 0 .
\end{align}
According to \eqref{eq: equation for normalized V_T}, we know that $g$ satisfies
\begin{align}
  \label{eq: equation for g}
	g(t,\theta, x)
  = \int_0^{\theta} \Pi_x^{(\phi)} [ (\phi^{-1}f)(\xi_t) e^{-\frac{1}{\gamma_0 - 1} J_g(t,r,\xi) } ] dr,
	\quad t\geq 0, \theta \geq 0, x\in E,
\end{align}
where, for each $t\geq 0$ and $r\geq 0$,
\begin{align}
  \label{eq: definition for J_g}
	J_g(t,r,\xi)
  :=(\gamma_0 - 1)t\int_0^1 \big(  \kappa \gamma \cdot (   \phi \eta_{ut}  )^{\gamma - 1} g (ut,ru^{\frac{1}{\gamma_0 - 1}},\cdot )^{\gamma-1}  \big) (  \xi_{(1-u)t}  ) du.
\end{align}
For each $t\geq 0$ and $r\geq 0$, define
\begin{align}
  \label{eq: definition of J'_G}
	J'_G(t,r,\xi)
  :=\gamma_0 (\gamma_0 - 1) t \int_0^1 \big( \mathbf 1_{\gamma(\cdot) = \gamma_0} \kappa \cdot (\phi \eta_{ut})^{\gamma_0 - 1} G\big( ru^{\frac{1}{\gamma_0 - 1}} \big) ^{\gamma_0 - 1} \big) (\xi_{(1-u)t}) du
\end{align}
and
\begin{equation}
  \label{eq: definition of J'_g}
	J'_g(t,r,\xi)
  := \gamma_0 (\gamma_0 - 1) t \int_0^1 \big( \mathbf 1_{\gamma(\cdot) = \gamma_0} \kappa \cdot (\phi \eta_{ut})^{\gamma_0 - 1} g\big( ut,ru^{\frac{1}{\gamma_0 - 1}}, \cdot \big)^{\gamma_0 - 1}  \big) (\xi_{(1-u)t})  du.
\end{equation}
The main idea is to show that $J_G,J'_G,J_g$ and $J'_g$ are approximately equal in some sense when $t\to \infty$.

Step 1:
We will give upper bounds for $G,g, J_G, J'_G, J_g$ and $J'_g$ respectively.
From \eqref{eq: equation for G} we have
\begin{align}
  \label{eq: upper bound for G}
	G(r)
	\leq r,
	\quad r \geq 0.
\end{align}	
From \eqref{eq: definition for J_G} and \eqref{eq: upper bound for G}, we have
\begin{align}
  \label{eq: upper bound for J_G}
	J_G(r)
	\leq \gamma_0 r^{\gamma_0 - 1},
	\quad r \geq 0.
\end{align}
From \eqref{eq: equation for g}, we have
\begin{align}
  \label{eq: upper bound for g}
	g(t,r, x) \leq c_f r,
	\quad t\geq 0, r \geq 0, x\in E.
\end{align}
From \eqref{eq: definition of eta}, \eqref{eq: definition for J_g}, \eqref{eq: upper bound for g} and the fact that $\gamma(\cdot) - 1 < 1$, we have $\Pi^{(\phi)}_{\cdot}$-almost surely
\begin{align}
	J_g(t,r, \xi)
	& \leq \|\kappa \cdot (c_f\phi)^{\gamma - 1} \|_\infty \int_0^1 \big(  t\eta_{ut}^{\gamma - 1} (ru^{\frac{1}{\gamma_0 - 1}} )^{\gamma-1}  \big) \big(  \xi_{(1-u)t} \big) du
	\\ & = \| \kappa \cdot (c_f\phi)^{\gamma - 1} \|_\infty \int_0^1 \big(  r^{\gamma - 1}t^{1-\frac{\gamma - 1}{\gamma_0 - 1}}  \big( C_X (\gamma_0 - 1) \big)^{-\frac{\gamma - 1}{\gamma_0 - 1}}  \big) \big( \xi_{(1-u)t} \big) du
	\\ & \leq \max\{1,r\} \cdot \| \kappa \cdot (c_f\phi)^{\gamma - 1} \|_\infty \Big\|  \big( C_X (\gamma_0 - 1) \big)^{-\frac{\gamma - 1}{\gamma_0 - 1}}\Big\|_\infty
	\\ & =: c_2 \cdot \max  \{1,r\},
       \quad t\geq 1, r\geq 0.
\end{align}
From \eqref{eq: definition of eta}, \eqref{eq: definition of J'_g} and \eqref{eq: upper bound for g}, we have $\Pi_{\cdot}^{(\phi)}$-almost surely
\begin{align}
	J'_g(t,r,\xi)
	& \leq \gamma_0 (\gamma_0 - 1) t \int_0^1 \big( \mathbf 1_{\gamma(\cdot) = \gamma_0} \kappa \cdot (\phi \eta_{ut})^{\gamma_0 - 1} (c_f ru^{\frac{1}{\gamma_0 - 1}})^{\gamma_0 - 1}\big) (\xi_{(1-u)t}) du
	\\ & \leq \gamma_0(\gamma_0 - 1) c_f^{\gamma_0 - 1}r^{\gamma_0 - 1} \|  \mathbf 1_{\gamma(\cdot) = \gamma_0}  \kappa \phi^{\gamma_0 - 1} \|_\infty \int_0^1 t \big( C_X(\gamma_0 - 1) ut \big)^{- 1} u du
	\\ & =: c_3 \cdot r^{\gamma_0 - 1},
       \quad t\geq 0, r\geq 0.
\end{align}
From \eqref{eq: definition of eta}, \eqref{eq: definition of J'_G} and \eqref{eq: upper bound for G}, we have $\Pi^{(\phi)}_\cdot$-almost surely
\begin{align}
  \label{eq: upper bound for J'_G}
	J'_G(t,r,\xi)
	& \leq \gamma_0 (\gamma_0 - 1) t \int_0^1 \big( \mathbf 1_{\gamma(\cdot) = \gamma_0} \kappa \cdot (\phi \eta_{ut})^{\gamma_0 - 1} (ru^{\frac{1}{\gamma_0 - 1}})^{\gamma_0 - 1}\big) (\xi_{(1-u)t}) du
	\\ & \leq \gamma_0(\gamma_0 - 1) r^{\gamma_0 - 1} \big\|   \mathbf 1_{\gamma(\cdot) = \gamma_0}  \kappa \phi^{\gamma_0 - 1} \big\|_\infty \int_0^1 t \big(  C_X(\gamma_0 - 1) ut \big)^{- 1}  u  du
	\\ & =: c_4 \cdot r^{\gamma_0 - 1},
       \quad t\geq 0, r\geq 0.
\end{align}

In the remainder of this proof, we use the following notation:
If $f$ is a measurable function which is $L^p$ integrable on the measure space $(S,\mathscr S,\mu)$ with $p > 0$, then we write
\begin{align}
	\|f\|_{\mu;p}
	:= \Big(\int_{S} |f|^p d\mu \Big)^{\frac{1}{p}}.
\end{align}
Notice that, when $p\geq 1$, $\|f\|_{\mu;p}$ is simply the $L^p$ norm of $f$ with respect to the measure $\mu$.	
However, when $p \in (0,1)$, $\|\cdot\|_{\mu; p}$ is not a norm.

Step 2: 
We will show that, for each $t\geq 0, \theta \geq 0,$ and $x\in E$
\begin{align}
	&| G(\theta)^{\gamma_0 - 1} - g(t,\theta,x)^{\gamma_0 - 1} |
	\\ & \quad \leq I_1(t,\theta,x) +c^{\gamma_0 - 1}_f I_2(t,\theta,x) + c^{\gamma_0 - 1}_f I_3(t,\theta,x) + c^{\gamma_0 - 1}_f I_4(t,\theta,x),
\end{align}
where
\begin{align}
    I_1(t,\theta,x)
    &:= \Big\| e^{ - J_G(r)} - \| (\phi^{-1}f)(\xi_t)^{\gamma_0 - 1} e^{-J_G(r)} \|_{\Pi_x^{(\phi)};\frac{1}{\gamma_0 - 1}} \Big\|_{\mathbf 1_{0\leq r\leq \theta} dr;\frac{1}{\gamma_0 - 1}} ,
    \\I_2(t,\theta,x)
    &:= \Big\|  \|  J_G(r) - J'_G(t,r,\xi)  \|_{\Pi_x^{(\phi)};\frac{1}{\gamma_0 - 1}} \Big\|_{\mathbf 1_{0\leq r\leq \theta} dr;\frac{1}{\gamma_0 - 1}},
    \\I_3(t,\theta,x)
    &:= \Big\| \|  J'_G(t,r,\xi) - J'_g(t,r,\xi)  \|_{\Pi_x^{(\phi)};\frac{1}{\gamma_0 - 1}} \Big\|_{\mathbf 1_{0\leq r\leq \theta} dr;\frac{1}{\gamma_0 - 1}},
\end{align}
and
\begin{align}
	I_4(t,\theta,x)
	:= \Big\| \| J'_g(t,r,\xi) - J_g(t,r,\xi)  \|_{\Pi_x^{(\phi)};\frac{1}{\gamma_0 - 1}} \Big\|_{\mathbf 1_{0\leq r\leq \theta} dr;\frac{1}{\gamma_0 - 1}}.
\end{align}
In fact, we can rewrite \eqref{eq: equation for G} and \eqref{eq: equation for g} as:
\begin{align}
	G(\theta)^{\gamma_0 - 1} =
	\| e^{ - J_G(r)} \|_{\mathbf 1_{0\leq r\leq \theta} dr;\frac{1}{\gamma_0 - 1}},
	\quad \theta \geq 0,
\end{align}	
and
\begin{align}
	g(t,\theta,x)^{\gamma_0 - 1}
	= \Big\| \| (\phi^{-1}f)(\xi_t) ^{\gamma_0 - 1} e^{-J_g(t,r,\xi)} \|_{\Pi_x^{(\phi)};\frac{1}{\gamma_0 - 1}} \Big\|_{\mathbf 1_{0\leq r\leq \theta} dr;\frac{1}{\gamma_0 - 1}},
	\quad t\geq 0, \theta \geq 0, x\in E.
\end{align}	
Therefore, by Minkowski's inequality we have that, for each $t\geq 0, \theta \geq 0$ and $x\in E$,
\begin{align}
	& | G(\theta)^{\gamma_0 - 1} - g(t,\theta,x)^{\gamma_0 - 1} |
	\\ & \quad \leq \Big\| e^{ - J_G(r)} - \| (\phi^{-1}f)(\xi_t)^{\gamma_0 - 1} e^{-J_g(t, r,\xi)} \|_{\Pi_x^{(\phi)};\frac{1}{\gamma_0 - 1}} \Big\|_{\mathbf 1_{0\leq r\leq \theta} dr;\frac{1}{\gamma_0 - 1}}
	\\ & \quad \leq I_1(t,\theta,x) + \Big\| \| (\phi^{-1}f)(\xi_t)^{\gamma_0 - 1} e^{-J_G(r)} \|_{\Pi_x^{(\phi)};\frac{1}{\gamma_0 - 1}} -
	\\ & \quad \qquad \qquad \qquad \qquad \qquad \| (\phi^{-1}f)(\xi_t)^{\gamma_0 - 1} e^{-J_g(t,r,\xi)} \|_{\Pi_x^{(\phi)};\frac{1}{\gamma_0 - 1}} \Big\|_{\mathbf 1_{0\leq r\leq \theta} dr;\frac{1}{\gamma_0 - 1}}
	\\ & \quad \leq I_1(t,\theta,x) + \Big\| \|  (\phi^{-1}f)(\xi_t)^{\gamma_0 - 1} ( e^{-J_G(r)} - e^{-J_g(t,r,\xi)} )  \|_{\Pi_x^{(\phi)};\frac{1}{\gamma_0 - 1}} \Big\|_{\mathbf 1_{0\leq r\leq \theta} dr;\frac{1}{\gamma_0 - 1}}
	\\ & \quad \leq I_1(t,\theta,x) + c_f^{\gamma_0 - 1}\Big\| \|  J_G(r) -J_g(t,r,\xi)  \|_{\Pi_x^{(\phi)};\frac{1}{\gamma_0 - 1}} \Big\|_{\mathbf 1_{0\leq r\leq \theta} dr;\frac{1}{\gamma_0 - 1}}
	\\ & \quad \leq I_1(t,\theta,x) + c_f^{\gamma_0 - 1} I_2(t,\theta,x) +c_f^{\gamma_0 - 1} I_3(t,\theta,x)+c_f^{\gamma_0 - 1} I_4(t,\theta,x).
\end{align}	

Step 3: 
We will show that, for each $\theta \geq 0$ and $x\in E$, $I_1(t,\theta,x) \xrightarrow[t\to \infty]{} 0$.
Notice that, by \eqref{eq:q(t,x,y)} and since $\langle f,\phi^* \rangle_m = 1$,
\begin{align}
	\Pi_x^{(\phi)} [(\phi^{-1}f)(\xi_t)]
	= \phi(x)^{-1}\Pi_x[f(\xi_t) e^{- \int_0^t \beta(\xi_s) ds}]
	= \phi(x)^{-1} P^\beta_t f(x)
	\xrightarrow[t\to \infty]{} 1,
	\quad x\in E.
\end{align}
Therefore,
\begin{align}
	& e^{ - J_G(r)} - \| (\phi^{-1}f)(\xi_t)^{\gamma_0 - 1} e^{-J_G(r)} \|_{\Pi_x^{(\phi)};\frac{1}{\gamma_0 - 1}}
	\\ & \quad =e^{ - J_G(r)} \Big( 1   -  \Pi_x^{(\phi)}[ (\phi^{-1}f)(\xi_t) ]^{\gamma_0 - 1}   \Big)
       \xrightarrow[t\to \infty]{} 0,
       \quad x\in E, r\geq 0.
\end{align}
We also have the following bound:
\begin{align}
	\Big| e^{ - J_G(r)} - \| (\phi^{-1}f)(\xi_t)^{\gamma_0 - 1} e^{-J_G(r)} \|_{\Pi_x^{(\phi)};\frac{1}{\gamma_0 - 1}} \Big|
	\leq 1+ c_f^{\gamma_0 - 1}.
\end{align}
Therefore, by the bounded convergence theorem, we have that, for each $\theta \geq 0$ and $x\in E$, $I_1(t,\theta, x) \xrightarrow[t\to \infty]{} 0$.

Step 4: 
We will show that, for each $\theta \geq 0$ and $x\in E$, $I_2(t,\theta,x) \xrightarrow[t\to \infty]{} 0$.
Notice that, according to \eqref{eq: definition for J_G} and \eqref{eq: definition of J'_G}, for each $t\geq 0$ and $r\geq 0$,
\begin{align}
	& J_G(r) - J'_G(t,r,\xi)
	\\ & \quad = \int_0^1 \gamma_0 G\big( ru^{\frac{1}{\gamma_0 - 1}} \big) ^{\gamma_0 - 1} \big( 1- (\gamma_0 - 1) \mathbf 1_{\gamma(\cdot) = \gamma_0} \kappa \phi^{\gamma_0 - 1} tu\eta_{ut}^{\gamma_0 - 1} \big)(\xi_{(1-u)t}) \frac{du}{u}
	\\ & \quad = \int_0^1 \gamma_0 G\big( ru^{\frac{1}{\gamma_0 - 1}} \big) ^{\gamma_0 - 1} \big( 1- C_X^{-1}\mathbf 1_{\gamma(\cdot) = \gamma_0} \kappa \phi^{\gamma_0 - 1} \big)(\xi_{(1-u)t}) \frac{du}{u}.
\end{align}
Also notice that, according to \eqref{eq: upper bound for G}, for each $r \geq 0$, $u\in [0,1]$ and $x\in E$,
\begin{align}
	& \big| \gamma_0 G\big( ru^{\frac{1}{\gamma_0 - 1}} \big) ^{\gamma_0 - 1} \big( 1- C_X^{-1}\mathbf 1_{\gamma(\cdot) = \gamma_0} \kappa \phi^{\gamma_0 - 1} \big)(x) \frac{1}{u} \big|
	\\ & \quad \leq \frac{\gamma_0}{u} G\big( ru^{\frac{1}{\gamma_0 - 1}} \big) ^{\gamma_0 - 1} \big|\big( 1- C_X^{-1}\mathbf 1_{\gamma(\cdot) = \gamma_0} \kappa \phi^{\gamma_0 - 1} \big)(x) \big|
	\\ & \quad \leq \gamma_0r^{\gamma_0 - 1} \big( 1+ \big\|C_X^{-1}\mathbf 1_{\gamma(\cdot) = \gamma_0} \kappa \phi^{\gamma_0 - 1} \big\|_\infty \big).
\end{align}
Therefore, according to Lemma \ref{lem: ergodicity of the underlying process} and the definition of $C_X$, we have that, for each $r\geq 0$ and $x\in E$,
\begin{align}
	J_G(r) - J'_G(t,r,\xi)
	\xrightarrow[t\to \infty]{L^2(\Pi_x^{(\phi)})}
	\int_0^1 \frac{\gamma_0}{u} G\big( ru^{\frac{1}{\gamma_0 - 1}} \big) ^{\gamma_0 - 1} \big\langle 1- C_X^{-1}\mathbf 1_{\gamma(\cdot) = \gamma_0} \kappa \phi^{\gamma_0 - 1}, \phi\phi^*\big\rangle_m du
	=0.
\end{align}
According to \eqref{eq: upper bound for J_G} and \eqref{eq: upper bound for J'_G}, we have that, for each $r\geq 0$ and $t\geq 0$,
\begin{align} 
  \label{eq: upper bound for J_G - J'_G}
	\big| J_G(r) - J'_G(t,r,\xi)\big|
	\leq (\gamma_0 + c_4) r^{\gamma_0 - 1}.
\end{align}
Therefore, according to the bounded convergence theorem, we have that, for each $r\geq 0$ and $x\in E$,
\begin{align}
  \big\|  J_G(r) - J'_G(t,r,\xi)  \big\|_{\Pi_x^{(\phi)};\frac{1}{\gamma_0 - 1}}
  \xrightarrow[t\to \infty]{} 0.
\end{align}
According to \eqref{eq: upper bound for J_G - J'_G}, we have that, for each $\theta \geq 0$, $r\in [0,\theta]$ and $x\in E$,
\begin{align}
	\big\|  J_G(r) - J'_G(t,r,\xi)  \big\|_{\Pi_x^{(\phi)};\frac{1}{\gamma_0 - 1}}
	\leq (\gamma_0 + c_4) \theta^{\gamma_0 - 1}.
\end{align}
Finally, according to the bounded convergence theorem, we have that, for each $\theta\geq 0$ and $x\in E$, $I_2(t,\theta,x)\xrightarrow[t\to \infty]{} 0$.

Step 5: 
We will show that, for each $\theta \geq 0$ and $x\in E$, $I_4(t,\theta,x) \xrightarrow[t\to \infty]{} 0$.
We first note that, for each $t\geq 0$ and $r\geq 0$, we have
\begin{align}
  \label{eq: expression for J_g - J'_g}
	J_g(t,r,\xi) - J'_g(t,r,\xi)
	= (\gamma_0 - 1)t\int_0^1 \big( \mathbf 1_{\gamma(\cdot )> \gamma_0}  \kappa\gamma \cdot (   \phi \eta_{ut} )^{\gamma - 1} g (ut,ru^{\frac{1}{\gamma_0 - 1}},\cdot )^{\gamma-1}  \big) \big(  \xi_{(1-u)t} \big) du.
\end{align}
We then note that, according \eqref{eq: upper bound for g} and the definition of $\eta_t$, for each $r\geq 0$, $u\in (0,1)$ and $x\in E$, we have
\begin{align}
  \label{eq: integer in the expression of J_g - J'_g convergences to 0}
	& (\gamma_0 - 1)t  \mathbf 1_{\gamma(x)> \gamma_0}  \kappa(x)\gamma(x) \big(   \phi(x) \eta_{ut}   \big)^{\gamma(x) - 1} g \big(ut,ru^{\frac{1}{\gamma_0 - 1}},x \big)^{\gamma(x)-1}
	\\ & \quad \leq (\gamma_0 - 1) \big\| \kappa \gamma \cdot (c_f r\phi)^{\gamma - 1}\big\|_\infty \mathbf 1_{\gamma(x) > \gamma_0} t \eta_{ut}^{\gamma(x) - 1} u^{\frac{\gamma(x) - 1}{\gamma_0 - 1}}
	\\ & \quad = (\gamma_0 - 1) \big\| \kappa \gamma \cdot (c_f r\phi)^{\gamma - 1}\big\|_\infty \mathbf 1_{\gamma(x) > \gamma_0} t \big( C_X(\gamma_0 - 1) ut\big)^{-\frac{\gamma(x) - 1}{\gamma_0 - 1}} u^{\frac{\gamma(x) - 1}{\gamma_0 - 1}}
	\\ & \quad \leq (\gamma_0 - 1) \mathbf 1_{\gamma(x) > \gamma_0} t^{1-\frac{\gamma(x) - 1}{\gamma_0 - 1}}\big\| \kappa \gamma \cdot (c_f r\phi)^{\gamma - 1}\big\|_\infty \sup_{x\in E} \big( C_X(\gamma_0 - 1) \big)^{-\frac{\gamma(x) - 1}{\gamma_0 - 1}}
	\\ & \quad \xrightarrow[t\to \infty]{} 0.
\end{align}
This also gives an upper bound: For each $r\geq 0$, $u \in (0,1)$, $x\in E$ and $t\geq 1$, we have
\begin{align}
  \label{eq: upper bound for the integrator of J_g - J'_g}
	& (\gamma_0 - 1)t  \mathbf 1_{\gamma(x)> \gamma_0}  \kappa(x)\gamma(x) \big(   \phi(x) \eta_{ut}   \big)^{\gamma(x) - 1} g \big(ut,ru^{\frac{1}{\gamma_0 - 1}},x \big)^{\gamma(x)-1}
	\\ & \quad \leq (\gamma_0 - 1) \big\| \kappa \gamma \cdot (c_f r\phi)^{\gamma - 1}\big\|_\infty \sup_{x\in E} \big( C_X(\gamma_0 - 1) \big)^{-\frac{\gamma(x) - 1}{\gamma_0 - 1}}.
\end{align}
Now, with \eqref{eq: expression for J_g - J'_g}, \eqref{eq: integer in the expression of J_g - J'_g convergences to 0} and \eqref{eq: upper bound for J_g - J'_g}, we can apply Lemma \ref{lem: ergodicity of the underlying process} to the function
\begin{align}
	(y,u,t)
  \mapsto (\gamma_0 - 1)t  \mathbf 1_{\gamma(y)> \gamma_0}  \kappa(y)\gamma(y) \big(   \phi(y) \eta_{ut}   \big)^{\gamma(y) - 1} g \big(ut,ru^{\frac{1}{\gamma_0 - 1}},y \big)^{\gamma(y)-1},
\end{align}
which says that, for each $r\geq 0$,
\begin{align}
	J_g(t,r,\xi) - J'_g(t,r,\xi)
  \xrightarrow[t\to \infty]{L^2(\Pi_x^{(\phi)})} 0.
\end{align}
According to \eqref{eq: expression for J_g - J'_g} and \eqref{eq: upper bound for the integrator of J_g - J'_g}, for each $r\geq 0$ and $t\geq 1$, we have that
\begin{equation}
  \label{eq: upper bound for J_g - J'_g}
	\big |	J_g(t,r,\xi) - J'_g(t,r,\xi)  \big |
	\leq (\gamma_0 - 1) \big\| \kappa \gamma \cdot (c_f r\phi)^{\gamma - 1}\big\|_\infty \sup_{x\in E} \big( C_X(\gamma_0 - 1) \big)^{-\frac{\gamma(x) - 1}{\gamma_0 - 1}}.
\end{equation}
Therefore, according to the bounded convergence theorem, for each $r\geq 0$ and $x\in E$, we have that
\begin{align}
	\big\| J'_g(t,r,\xi) - J_g(t,r,\xi)  \big\|_{\Pi_x^{(\phi)};\frac{1}{\gamma_0 - 1}}
	\xrightarrow[t\to \infty]{} 0.
\end{align}
According to \eqref{eq: upper bound for J_g - J'_g}, for each $\theta \geq 0$, $r\in [0,\theta]$ , $t\geq 1$ and $x\in E$, we have that
\begin{align}
	\big\| J'_g(t,r,\xi) - J_g(t,r,\xi)  \big\|_{\Pi_x^{(\phi)};\frac{1}{\gamma_0 - 1}}
	\leq (\gamma_0 - 1) \big\| \kappa \gamma\cdot (c_f \theta \phi)^{\gamma - 1}\big\|_\infty \sup_{x\in E} \big( C_X(\gamma_0 - 1) \big)^{-\frac{\gamma(x) - 1}{\gamma_0 - 1}}.
\end{align}
Therefore, according to the bounded convergence theorem, for each $\theta \geq 0$ and $x\in E$, we have that $I_4(t,\theta, x) \xrightarrow[t\to \infty]{} 0$.

Step 6: 
We will show that
\begin{align}
  \limsup_{t\to \infty} I_3(t,\theta,x)
  \leq \gamma_0 \Big(  \int_0^\theta  \| M(r u^{\frac{1}{\gamma_0 - 1}}) \|_{\mathbf 1_{0\leq u\leq 1}\frac{du}{u};\gamma_0 - 1}  dr\Big)^{\gamma_0 - 1},
  \quad \theta \geq 0, x\in E,
\end{align}
where
\begin{align}
	M(t,r,x)
	:= |G(r)^{\gamma_0 - 1} - g(t,r,x)^{\gamma_0 - 1}|^{\frac{1}{\gamma_0 - 1}},
	\quad t\geq 0, r\geq 0, x\in E,
\end{align}
and
\begin{align}
	M(r,x)
	:= \limsup_{t\to \infty} M(t,r,x);
	\quad M(r):= \sup_{x\in E} M(r,x),
	\quad r\geq 0, x\in E.
\end{align}
Notice that, according to \eqref{eq: upper bound for G} and \eqref{eq: upper bound for g}, we have the following bound:
\begin{align}
  \label{eq: upper bound for M(t,r,x)}
	M(t,r,x)
	\leq |r^{\gamma_0 - 1} + c_f^{\gamma_0 - 1} r^{\gamma_0 - 1} | ^{\frac{1}{\gamma_0 - 1}}
	=: c_6 r,
\end{align}
where the constant $c_6$ is independent of $t$ and $x$.
Therefore, we have
\begin{align}
	M(r,x)
	\leq M(r)
	\leq c_6 r,
	\quad r\geq 0, x\in E.
\end{align}	
From the definition of $J'_G, J'_g$ and $\eta_t$, we have for each $t \geq 0$ and $r\geq 0$,
\begin{align}
	\label{eq: differences between J'_G and J'_g}
	& |J'_G(t,r,\xi) - J'_g(t,r,\xi)|
	\\ & \quad \leq \gamma_0(\gamma_0 - 1) t \int_0^1 \big( \mathbf 1_{\gamma(\cdot) = \gamma_0} \kappa \cdot (\phi \eta_{ut})^{\gamma_0 - 1} M(ut,ru^{\frac{1}{\gamma_0 - 1}},\cdot)^{\gamma_0 - 1}\big)(\xi_{(1-u)t}) du
	\\ & \quad = \gamma_0 C_X^{-1}\int_0^1 \big( \mathbf 1_{\gamma(\cdot) = \gamma_0} \kappa  \phi^{\gamma_0 - 1}  u^{-1} M(ut,ru^{\frac{1}{\gamma_0 - 1}},\cdot)^{\gamma_0 - 1}\big)(\xi_{(1-u)t}) du.
\end{align}
According to \eqref{eq: upper bound for M(t,r,x)}, we have the following upper bound:
\begin{align}
  u^{-1} M(ut,ru^{\frac{1}{\gamma_0 - 1}}, x)
  \leq c_6 ru^{\frac{2-\gamma_0}{\gamma_0 - 1}}
  \leq c_6 r,
  \quad u\in (0,1), r\geq 0, t\geq 0, x\in E.
\end{align}
Therefore, fixing an $r\geq 0$, we can apply Lemma \ref{lem: Fatou-ergodic lemma for the uderlying process} to the function
\begin{align}
	(y,u,t)
	\mapsto \gamma_0 C_X^{-1}\mathbf 1_{\gamma(y) = \gamma_0} \kappa(y)  \phi(y)^{\gamma_0 - 1}  u^{-1} M(ut,ru^{\frac{1}{\gamma_0 - 1}},y)^{\gamma_0 - 1}
\end{align}
since it is a bounded Borel function on $E\times (0,1) \times [0,\infty)$.
Now, according to Lemma \ref{lem: Fatou-ergodic lemma for the uderlying process}, \eqref{eq: differences between J'_G and J'_g} and the definitions of $M(r,x), M(r)$ and $C_X$, we have
\begin{align}
  \label{eq: limsup of J_G - J'_g}
	& \limsup_{t\to \infty} \| J_G'(t,r,\xi) - J'_g(t,r,\xi) \|_{\Pi_x^{\phi};\frac{1}{\gamma_0 - 1}}
	\\ & \quad \leq  \gamma_0 C_X^{-1} \int_0^1 \big\langle \mathbf 1_{\gamma(\cdot) = \gamma_0} \kappa \phi^{\gamma_0 - 1} M(ru^{\frac{1}{\gamma_0 - 1}},\cdot)^{\gamma_0 - 1}, \phi\phi^* \big\rangle_m \frac{du}{u}
	\\ & \quad \leq  \gamma_0  \int_0^1  M(ru^{\frac{1}{\gamma_0 - 1}})^{\gamma_0 - 1} \frac{du}{u}.
\end{align}
We recall the reverse Fatou's lemma in $L^p$ with $p\geq 1$: Let $(f_n)_{n\in \mathbb N}$ be a sequence of non-negative measurable functions defined on a measure space $S$ with $\sigma$-finite measure $\mu$. If there exists a non-negative $L^p(\mu)$-integrable function $g$ on $S$ such that $f_n \leq g$ for all $n$, then according to the classical reverse Fatou's lemma, we have
\begin{align}
	\limsup_{n\to \infty}\big\| f_n \big\|_{\mu;p}
	= \Big ( \limsup_{n\to \infty}  \int f^p_n d\mu \Big)^{\frac{1}{p}}
	\leq  \Big ( \int \limsup_{n\to \infty} f^p_n d\mu \Big)^{\frac{1}{p}}
	= \big\| \limsup_{n\to \infty} f_n \big\|_{\mu;p}.
\end{align}
Using this version of the reverse Fatou's lemma and \eqref{eq: limsup of J_G - J'_g}, we get that
\begin{align}
	& \limsup_{t\to \infty} I_3(t,\theta, x)
   \leq \big\| \limsup_{t\to \infty} \|    J'_G(t,r,\xi) - J'_g(t,r,\xi) \|_{\Pi_x^{(\phi)};\frac{1}{\gamma_0 - 1}} \big\|_{\mathbf 1_{0\leq r\leq \theta} dr;\frac{1}{\gamma_0 - 1}}
	\\ & \quad\leq \Big\| \gamma_0  \int_0^1  M(ru^{\frac{1}{\gamma_0 - 1}})^{\gamma_0 - 1} \frac{du}{u} \Big\|_{\mathbf 1_{0\leq r\leq \theta} dr;\frac{1}{\gamma_0 - 1}}
	\\ & \quad = \gamma_0 \bigg( \int_0^\theta \Big (   \int_0^1  M(ru^{\frac{1}{\gamma_0 - 1}})^{\gamma_0 - 1} \frac{du}{u}   \Big )^{\frac{1}{\gamma_0 - 1}} dr \bigg)^{\gamma_0 - 1}
	\\ & \quad = \gamma_0 \Big(  \int_0^\theta  \| M(r u^{\frac{1}{\gamma_0 - 1}}) \|_{\mathbf 1_{0\leq u\leq 1}\frac{du}{u};\gamma_0 - 1}  dr\Big)^{\gamma_0 - 1},
	\quad \theta \geq 0, x\in E.
\end{align}
	
Step 7: 
We will show that $M(\theta) = 0$ for each $\theta \geq 0$.
We first claim that
\begin{align}
	M(\theta)
	\leq c_M\int_0^\theta  \big\| M(r u^{\frac{1}{\gamma_0 - 1}}) \big\|_{\mathbf 1_{0\leq u\leq 1}\frac{du}{u};\gamma_0 - 1}  dr ,
	\quad \theta \geq 0,
\end{align}
for some constant $c_M > 0$.
In fact, a direct application of Steps 2-6 gives that, for each $t\geq 0$ and $x\in E$:
\begin{align}
	& M(r,x)^{\gamma_0 - 1}
   = \limsup_{t\to \infty} M(t,r,x)^{\gamma_0 - 1}
   = \limsup_{t\to \infty}|G(r)^{\gamma_0 - 1} - g(t,r,x)^{\gamma_0 - 1}|
	\\ & \quad \leq \limsup_{t\to \infty} \big( I_1(t,\theta,x) +c^{\gamma_0 - 1}_f I_2(t,\theta,x) +c^{\gamma_0 - 1}_f I_3(t,\theta,x) + c^{\gamma_0 - 1}_f I_4(t,\theta,x) \big)
	\\ & \quad = c_f^{\gamma_0 - 1} \limsup_{t\to \infty} I_3(t,\theta ,x)
       \leq c_f^{\gamma_0 - 1} \gamma_0 \Big(  \int_0^\theta  \big\| M(r u^{\frac{1}{\gamma_0 - 1}}) \big\|_{\mathbf 1_{0\leq u\leq 1}\frac{du}{u};\gamma_0 - 1}  dr\Big)^{\gamma_0 - 1}.
\end{align}
Therefore, for each $\theta \geq 0$,
\begin{align}
	M(\theta)
	= \sup_{x\in E}  M(r,x)
	\leq c_f \gamma_0^{\frac{1}{\gamma_0 - 1}} \int_0^\theta  \big\| M(r u^{\frac{1}{\gamma_0 - 1}}) \big\|_{\mathbf 1_{0\leq u\leq 1}\frac{du}{u};\gamma_0 - 1}  dr.
\end{align}
According to that $M(\theta) \leq c_6 \theta$ for each $\theta$, we can apply Lemma \ref{lem: F is zero} to the above inequality to get the desired result.

Step 8: 
Finally, $M \equiv 0$ clearly implies that $\lim_{t\to \infty} I_3(t,\theta, x) = 0$, and thus completes the verification of \eqref{eq: we only need to proof this}.

\appendix
\section{}
\subsection{Examples}
\label{sec:examples}

In this Subsection, we briefly recall from \cite{RenSongZhang2015Limit} some examples of Markov processes satisfying Assumptions \ref{asp: 1} and \ref{asp: 3}. 
We will not try to give the most general examples. 
For details and more examples, we refer our readers to \cite{RenSongZhang2015Limit}.

\begin{exa}
  Suppose that $E$ is a finite state space and $m$ is the counting measure on $E$. 
  Let $\xi$ be an irreducible, continuous-time Markov chain. 
  Then the semigroup $(P_t)_{t\ge 0}$ of $\xi$ satisfies Assumptions \ref{asp: 1} and \ref{asp: 3}.
\end{exa}

\begin{exa} 
  Suppose that $E$ is a bounded Lipschitz connected  open set of $\mathbb R^d$ and that $m$ denotes the Lebesgue measure on $E$.
  Let $\xi$ be the subprocess in $E$ of a diffusion process in $\mathbb{R}^d$ corresponding to a uniformly elliptic divergence form second order differential operator.  
  Then the semigroup $(P_t)_{t\ge 0}$ of $\xi$ satisfies Assumptions \ref{asp: 1} and \ref{asp: 3}.
\end{exa}

\begin{exa}
  Suppose that $E$ is the closure of a bounded $C^2$ connected  open set of $\mathbb R^d$ and that $m$ denotes the Lebesgue measure on $E$.
  Let $\xi$ be the reflecting Brownian motion in $E$. 
  Then the semigroup $(P_t)_{t\ge 0}$ of $\xi$ satisfies
  Assumptions \ref{asp: 1} and \ref{asp: 3}. 
\end{exa}

\begin{exa}
  Suppose that $E$ is a bounded   open set of $\mathbb R^d$ and  $m$ denotes the Lebesgue measure on $E$.
  $\xi$ be the subprocesses in $E$ of any of the subordinate Brownian motions studied in \cite{KSV1, KSV2}.  
  Then the semigroup $(P_t)_{t\ge 0}$ of $\xi$ satisfies Assumptions \ref{asp: 1} and \ref{asp: 3}. 
\end{exa}

\begin{exa}
  Suppose $a>2$ is a constant.
  Assume that $E=\mathbb{R}^d$ and $m$ is the Lebesgue measure on $\mathbb{R}^d$.
  Let $\xi$ be a Markov process on $\mathbb{R}^d$ corresponding to the infinitesimal generator $\Delta-|x|^a$. 
  Then the semigroup $(P_t)_{t\ge 0}$ of $\xi$ satisfies Assumption \ref{asp: 1} and \ref{asp: 3}. 
\end{exa}

\begin{exa} 
  Assume that $E=\mathbb{R}^d$ and $m$ is the Lebesgue measure on $\mathbb{R}^d$.
  Suppose that $V$ is a nonnegative and locally bounded function on  $\mathbb{R}^d$ such that there exist $R>0$ and $M\ge 1$ such that for all $|x|>R$,
  \begin{align}
    M^{-1}(1+V(x))\le V(y)\le M(1+V(x)), \qquad y\in B(x, 1),
  \end{align}
  and that
  \begin{align}
    \lim_{|x|\to\infty}\frac{V(x)}{\log|x|}=\infty.
  \end{align}
  Suppose $\beta\in (0, 2)$ is a constant.
  Let $\xi$ be a Markov process on $\mathbb{R}^d$ corresponding to the infinitesimal generator $-(-\Delta)^{\beta/2}-V(x)$. 
  Then the semigroup $(P_t)_{t\ge 0}$ of $\xi$ satisfies
  Assumptions \ref{asp: 1} and \ref{asp: 3}. 
\end{exa}

\begin{exa}
  Suppose that $\beta\in (0, 2)$ and that $\xi^{(1)}=\{\xi^{(1)}_t: t\ge0\}$ is a strictly $\beta$-stable process in $\mathbb{R}^d$. 
  Suppose that, in the case $d\ge 2$, the spherical part $\eta$ of the L\'evy measure $\mu$ of $\xi^{(1)}$ satisfies the following assumption: 
  there exist a positive function $\Phi$ on the unit sphere $S$ in $\mathbb{R}^d$ and $\kappa>1$ such that
  \begin{align}
    \Phi=\frac{d\eta}{d\sigma} \quad \text{and} \quad
    \kappa^{-1}\le \Phi(z)\le \kappa \quad \text{on } S
  \end{align}
  where $\sigma$ is the surface measure on $S$. In the case $d=1$, we assume that the L\'evy
  measure of $\xi^{(1)}$ is given by
  \begin{align}
    \mu(dx)=c_1x^{-1-\beta}1_{\{x>0\}}+ c_2|x|^{-1-\beta}1_{\{x<0\}}
  \end{align}
  with $c_1, c_2>0$. Suppose that $E$ is a bounded open set in $\mathbb{R}^d$ and $m$ is the Lebesgue measure on $E$.
  Let $\xi$ be the process in $E$ obtained by killing $\xi^{(1)}$ upon exiting $E$.
  Then the semigroup $(P_t)_{t\ge 0}$ of $\xi$ satisfies
  Assumptions \ref{asp: 1} and \ref{asp: 3}.
\end{exa}

\subsection{Analytical results}
\label{sec: Characterizing the Zolotarev's distribution using an non-linear delay equation}
In this Subsection, we give the proofs of the two lemmas used in the proof of Theorem \ref{thm: main theorem}(3). We think these two lemmas are of independent interest.

We first recall the following notation:
If $f$ is a measurable function which is $L^p$ integrable on the measure space $(S,\mathscr S,\mu)$ with $p > 0$, then we write
\begin{align}
	\|f\|_{\mu;p}
	:= \Big(\int_{S} |f|^p d\mu \Big)^{\frac{1}{p}}.
\end{align}
Notice that, when $p\geq 1$, $\|f\|_{\mu;p}$ is simply the $L^p$ norm of $f$ with respect to the measure $\mu$.	
However, when $p \in (0,1)$, $\|\cdot\|_{\mu; p}$ is not a norm.

\begin{lem}
  \label{lem: F is zero}
  Suppose that $\alpha \in (1,2)$.
  Suppose that $F$ is a non-negative function on $[0,\infty)$ satisfying the property that there exists a constant $C>0$ such that  $F(\theta) \leq C\theta$ for all $\theta \geq 0$ and
  \begin{align}
    \label{eq:Gronwall_inequlity}
    F(\theta)
    \leq C \int_0^\theta \|  F(ru^{ \frac{1}{\alpha - 1}  })\|_{\mathbf 1_{0<u<1}\frac{du}{u}; \alpha - 1} dr, \quad \theta \geq 0.
  \end{align}
	Then $F \equiv 0$.
\end{lem}

\begin{proof}
  We claim that for each $k \in \mathbb N$, we have
  \begin{align}
    \label{eq:upperbound_for_F}
    F(\theta) \leq \frac{C^k \theta^k}{k!},
    \quad \theta \geq 0.
  \end{align}
  In fact, when $k = 1$ this is trival. Now if \eqref{eq:upperbound_for_F} is true for a fixed $k \in \mathbb N$, then from
  \begin{align}
    & F(\theta)
      \leq C \int_0^\theta \big\|  F(ru^{ \frac{1}{\alpha - 1}  })\big\|_{\mathbf 1_{0<u<1}\frac{du}{u}; \alpha - 1} dr
      \leq C \int_0^\theta \Big\|  \frac{1}{k!}(Cru^{ \frac{1}{\alpha - 1}  })^k \Big\|_{\mathbf 1_{0<u<1}\frac{du}{u}; \alpha - 1} dr
    \\ & \leq \frac{C^{k+1}}{k!} \Big(\int_0^\theta r^k dr\Big) \cdot \|u^{\frac{k}{\alpha - 1}} \|_{\mathbf 1_{0<u<1}\frac{du}{u}; \alpha - 1} \leq \frac{C^{k+1} \theta^{k+1}}{(k+1)!},
  \end{align}
  we have that \eqref{eq:upperbound_for_F} is true for $k+1$.
  Therefore, by induction, \eqref{eq:upperbound_for_F} is true for all $k \in \mathbb N$.
  
  Letting $k \to \infty$ in \eqref{eq:upperbound_for_F}, we get that $F(\theta) = 0$ for each $\theta \geq 0$.
\end{proof}

\begin{lem} 
  \label{lem: characterize the general Mittag-Leffler distribution}
  Suppose that $\alpha \in (1,2).$
  The non-linear delay equation
  \begin{align} 
    \label{eq: equation for the distribution}
    G( \theta)
    = \int_0^\theta \exp\Big\{ - \frac{\alpha} {\alpha - 1} \int_0^1 G(ru^{\frac{1}{\alpha - 1} })^{\alpha - 1}\frac{du}{u} \Big\} dr,
    \quad \theta \geq 0,
  \end{align}
	has a unique solution:
  \begin{align}
    \label{eq: solution for the equation}
    G(\theta)
    = \Big(\frac{1}{1+\theta^{-(\alpha - 1)}}\Big)^{\frac{1}{\alpha - 1}},
    \quad \theta \geq 0.
  \end{align}
\end{lem}

\begin{proof}
	We first verify that \eqref{eq: solution for the equation}  is a solution of \eqref{eq: equation for the distribution}.
	In fact, if $G(\theta) = (\frac{1}{1+ \theta^{-(\alpha - 1)}})^{\frac{1}{\alpha - 1}}$, then
  \begin{align}
    & \int_0^\theta \exp\Big\{- \frac{\alpha} {\alpha - 1} \int_0^1 G(ru^{\frac{1}{\alpha - 1}})^{\alpha - 1}\frac{du}{u}\Big\} dr
    \\ & \quad = \int_0^\theta \exp\Big\{- \frac{\alpha} {\alpha - 1} \int_0^1 \frac{du}{u+r^{-(\alpha - 1)}} \Big\} dr
         = \int_0^\theta \exp\Big\{- \frac{\alpha} {\alpha - 1} \log \frac {1+r^{-(\alpha - 1)}} {r^{-(\alpha - 1)} } \Big\} dr
    \\ & \quad = \int_0^\theta \big(\frac{1+r^{-(\alpha - 1)}}{r^{-(\alpha - 1)} }\big)^{- \frac{\alpha} {\alpha - 1}} dr
         = \int_0^\theta \big( 1 + r^{ - ( \alpha - 1 ) } \big)^{- \frac{\alpha} {\alpha - 1}} r^{-\alpha} dr
         = G(\theta).
  \end{align}
	The last equality is due to $G(0) = 0$ and
  \begin{align}
    \frac{d}{d\theta}G(\theta)
    & = - \frac{1}{\alpha - 1} \big(1+\theta^{-(\alpha - 1)}\big)^{- \frac{1}{\alpha - 1} - 1} \frac{d}{d\theta} \theta^{-(\alpha - 1)}
    \\ & =  \big(1+\theta^{-(\alpha - 1)}\big)^{- \frac{\alpha}{\alpha - 1} } \theta^{-\alpha}.
  \end{align}
	
  Now assume that $G_0$ is another solution to the equation \eqref{eq: equation for the distribution}, we then only have to show that $G_0 = G$.
  This can be done by showing that $F(\theta) = 0$ where
  \begin{align}
    F(\theta)
    := |G(\theta)^{\alpha - 1} - G_0(\theta)^{\alpha - 1}|^{\frac{1}{\alpha - 1}},
    \quad \theta \geq 0.
  \end{align}
  We claim that the non-negative function $F$ satisfies the inequality \eqref{eq:Gronwall_inequlity} with $C = \alpha ^{1/(\alpha - 1)}$.
	In fact, by the $L^p$ Minkowski inequality with $p = \frac{1}{\alpha - 1} > 1$, we have
  \begin{align}
    & |G(\theta)^{\alpha - 1} - G_0(\theta)^{\alpha - 1}|
    \\ & \quad = \Big| \|e^{-\alpha \int_0^1 G(ru^{\frac{1}{\alpha - 1}})^{\alpha - 1} \frac{du}{u}} \|_{\mathbf 1_{0<r<\theta} dr; \frac{1}{\alpha - 1}} - \|e^{-\alpha \int_0^1 G_0(ru^{ \frac{1} {\alpha - 1}})^{\alpha - 1} \frac{du}{u}} \|_{\mathbf 1_{0<r<\theta}dr;\frac{1}{\alpha - 1}} \Big|
    \\ & \quad \leq \| e^{-\alpha \int_0^1 G(ru^{\frac{1}{\alpha - 1}})^{\alpha - 1} \frac{du}{u}} - e^{-\alpha \int_0^1 G_0(ru^{\frac{1}{\alpha - 1}})^{\alpha - 1} \frac{du}{u}} \|_{\mathbf 1_{0<r<\theta}dr;\frac{1} {\alpha - 1}}
    \\ & \quad \leq \Big\| \alpha \int_0^1 G(ru^{\frac{1}{\alpha - 1}})^{\alpha - 1} \frac{du}{u} - \alpha \int_0^1 G_0(ru^{\frac{1}{\alpha - 1}})^{\alpha - 1} \frac{du}{u} \Big\|_{\mathbf 1_{0<r<\theta}dr;\frac{1} {\alpha - 1}}
    \\ & \quad \leq \alpha \Bigg( \int_0^\theta \Big( \int_0^1 |G(ru^{\frac{1}{\alpha - 1}})^{\alpha - 1} - G_0(ru^{\frac{1}{\alpha - 1}})^{\alpha - 1}| \frac{du}{u} \Big)^{\frac{1}{\alpha - 1}} dr \Bigg)^{\alpha - 1}.
  \end{align}
  This implies the claim.
  
  On the other hand, according to \eqref{eq: equation for the distribution}, we have that $G(\theta) \leq \theta$ and $G_0(\theta) \leq \theta$.
  Therefore, we also have that there is a constant $C_1 > 0$ such that $F(\theta) \leq C_1 \theta$.
  Therefore, according to Lemma \ref{lem: F is zero}, we have $F \equiv 0$ as desired.
\end{proof}

\subsection*{Acknowledgment:}
We thank the referees for helpful comments on the first version of this paper.


\end{document}